\documentclass{article}
\usepackage{amsmath}
\usepackage{amssymb}
\usepackage{amsthm}
\usepackage{cite}
\usepackage[arrow,curve,matrix,arc]{xy}
\begin{document}
\def\e#1\e{\begin{equation}#1\end{equation}}
\def\ea#1\ea{\begin{align}#1\end{align}}
\def\eq#1{{\rm(\ref{#1})}}
\theoremstyle{plain}
\newtheorem{thm}{Theorem}[section]
\newtheorem{lem}[thm]{Lemma}
\newtheorem{prop}[thm]{Proposition}
\newtheorem{cor}[thm]{Corollary}
\theoremstyle{definition}
\newtheorem{dfn}[thm]{Definition}
\newtheorem{ex}[thm]{Example}
\newtheorem{rem}[thm]{Remark}
\newtheorem{conv}[thm]{Convention}
\numberwithin{figure}{section}
\def\dim{\mathop{\rm dim}\nolimits}
\def\codim{\mathop{\rm codim}\nolimits}
\def\depth{\mathop{\rm depth}\nolimits}
\def\Ker{\mathop{\rm Ker}}
\def\Coker{\mathop{\rm Coker}}
\def\sign{\mathop{\rm sign}}
\def\GL{\mathop{\rm GL}\nolimits}
\def\Im{\mathop{\rm Im}}
\def\id{\mathop{\rm id}\nolimits}
\def\Or{\mathop{\rm Or}}
\def\Man{{\mathop{\bf Man}}}
\def\Manb{{\mathop{\bf Man^b}}}
\def\Manc{{\mathop{\bf Man^c}}}
\def\dManc{{\mathop{\bf dMan^c}}}
\def\ge{\geqslant}
\def\le{\leqslant\nobreak}
\def\R{{\mathbin{\mathbb R}}}
\def\Z{{\mathbin{\mathbb Z}}}
\def\N{{\mathbin{\mathbb N}}}
\def\RP{{\mathbin{\mathbb{RP}}}}
\def\oM{{\mathbin{\smash{\,\,\overline{\!\!\mathcal M\!}\,}}}}
\def\al{\alpha}
\def\be{\beta}
\def\ga{\gamma}
\def\de{\delta}
\def\io{\iota}
\def\ep{\epsilon}
\def\la{\lambda}
\def\ka{\kappa}
\def\th{\theta}
\def\ze{\zeta}
\def\up{\upsilon}
\def\vp{\varphi}
\def\si{\sigma}
\def\om{\omega}
\def\De{\Delta}
\def\La{\Lambda}
\def\Om{\Omega}
\def\Up{\Upsilon}
\def\Ga{\Gamma}
\def\Si{\Sigma}
\def\Th{\Theta}
\def\pd{\partial}
\def\ts{\textstyle}
\def\sst{\scriptscriptstyle}
\def\w{\wedge}
\def\sm{\setminus}
\def\bu{\bullet}
\def\bs{\boldsymbol}
\def\op{\oplus}
\def\od{\odot}
\def\op{\oplus}
\def\ov{\overline}
\def\bigop{\bigoplus}
\def\bigot{\bigotimes}
\def\iy{\infty}
\def\es{\emptyset}
\def\ra{\rightarrow}
\def\Ra{\Rightarrow}
\def\ab{\allowbreak}
\def\longra{\longrightarrow}
\def\hookra{\hookrightarrow}
\def\dashra{\dashrightarrow}
\def\t{\times}
\def\ci{\circ}
\def\ti{\tilde}
\def\d{{\rm d}}
\def\md#1{\vert #1 \vert}
\def\bmd#1{\big\vert #1 \big\vert}
\def\an#1{\langle #1 \rangle}
\title{On manifolds with corners}
\author{Dominic Joyce}
\date{}
\maketitle
\begin{abstract}
Manifolds without boundary, and manifolds with boundary, are
universally known and loved in Differential Geometry, but {\it
manifolds with corners\/} (locally modelled on
$[0,\iy)^k\t\R^{n-k}$) have received comparatively little attention.
The basic definitions in the subject are not agreed upon, there are
several inequivalent definitions in use of manifolds with corners,
of boundary, and of smooth map, depending on the applications in
mind.

We present a theory of manifolds with corners which includes a new
notion of smooth map $f:X\ra Y$. Compared to other definitions, our
theory has the advantage of giving a category $\Manc$ of manifolds
with corners which is particularly well behaved as a category: it
has products and direct products, boundaries $\pd X$ behave in a
functorial way, and there are simple conditions for the existence of
fibre products $X\t_ZY$ in $\Manc$.

Our theory is tailored to future applications in Symplectic
Geometry, and is part of a project to describe the geometric
structure on moduli spaces of $J$-holomorphic curves in a new way.
But we have written it as a separate paper as we believe it is of
independent interest.
\end{abstract}

\section{Introduction}
\label{mc1}

Most of the literature in Differential Geometry discusses only
manifolds without boundary (locally modelled on $\R^n$), and a
smaller proportion manifolds with boundary (locally modelled on
$[0,\iy)\t\R^{n-1}$). Only a few authors have seriously studied {\it
manifolds with corners} (locally modelled on $[0,\iy)^k\t\R^{n-k}$).
They were first developed by Cerf \cite{Cerf} and Douady \cite{Doua}
in 1961, who were primarily interested in their Differential
Geometry. J\"anich \cite{Jani} used manifolds with corners to
classify actions of transformation groups on smooth manifolds.
Melrose \cite{Melr1,Melr2} and others study analysis of elliptic
operators on manifolds with corners. Laures \cite{Laur} defines a
cobordism theory for manifolds with corners, which has been applied
in Topological Quantum Field Theory, by Lauda and Pfeiffer
\cite{LaPf} for instance. Margalef-Roig and Outerelo Dominguez
\cite{MaOu} generalize manifolds with corners to
infinite-dimensional Banach manifolds.

How one sets up the theory of manifolds with corners is not
universally agreed, but depends on the applications one has in mind.
As we explain in Remarks \ref{mc2rem}, \ref{mc3rem} and
\ref{mc6rem2} which relate our work to that of other authors, there
are at least four inequivalent definitions of manifolds with
corners, two inequivalent definitions of boundary, and (including
ours) four inequivalent definitions of smooth map in use in the
literature. The purpose of this paper is to carefully lay down the
foundations of a theory of manifolds with corners, which includes a
new notion of {\it smooth map\/} $f:X\ra Y$ between manifolds with
corners.

The main issue here is that (in our theory) an $n$-manifold with
corners $X$ has a boundary $\pd X$ which is an $(n\!-\!1)$-manifold
with corners, and so by induction the $k$-fold boundary $\pd^kX$ is
an $(n\!-\!k)$-manifold with corners. How to define smooth maps
$f:X\ra Y$ in the interiors $X^\ci,Y^\ci$ is clear, but one must
also decide whether to impose compatibility conditions on $f$ over
$\pd^kX$ and $\pd^lY$, and it is not obvious how best to do this.
Our definition gives a nicely behaved category $\Manc$ of manifolds
with corners, and in particular we can give simple conditions for
the existence of {\it fibre products\/} in~$\Manc$.

The author's interest in manifolds with corners has to do with
applications in Symplectic Geometry. Moduli spaces
$\oM_{g,h}^{l,m}(M,L,J,\be)$ of stable $J$-holomorphic curves in a
symplectic manifold $(M,\om)$ with boundary in a Lagrangian $L$ are
used in Lagrangian Floer cohomology, open Gromov--Witten invariants,
and Symplectic Field Theory. What geometric structure should we put
on $\oM_{g,h}^{l,m}(M,L,J,\be)$? Hofer, Wysocki and Zehnder
\cite{HWZ} make $\oM_{g,h}^{l,m}(M,L,J,\be)$ into a {\it polyfold\/}
(actually, the zeroes of a Fredholm section of a polyfold bundle
over a polyfold). Fukaya, Oh, Ohta and Ono \cite{FOOO} make
$\oM_{g,h}^{l,m}(M,L,J,\be)$ into a {\it Kuranishi space}. But the
theory of Kuranishi spaces is still relatively unexplored, and in
the author's view even the definition is not satisfactory.

In \cite{Joyc} the author will develop theories of {\it
d-manifolds\/} and {\it d-orbifolds}. Here {\it d-manifolds\/} are a
simplified version of Spivak's {\it derived manifolds\/}
\cite{Spiv}, a diff\-er\-ent\-ial-geo\-met\-ric offshoot of Jacob
Lurie's Derived Algebraic Geometry programme, and d-orbifolds are
the orbifold version of d-manifolds. We will argue that the
`correct' way to define Kuranishi spaces is as {\it d-orbifolds with
corners}, which will help to make \cite{FOOO} more rigorous. In
future the author hopes also to show that polyfolds can be truncated
to d-orbifolds with corners, building a bridge between the theories
of Fukaya et al.\ \cite{FOOO} and Hofer et al.~\cite{HWZ}.

To define d-manifolds and d-orbifolds with corners we first need a
theory of manifolds with corners, and we develop it here in a
separate paper as we believe it is of independent interest. For
\cite{Joyc} and later applications in Symplectic Geometry, it is
important that boundaries and fibre products in $\Manc$ should be
well-behaved. The author strongly believes that the theory we set
out here, in particular our definition of smooth map, is the `right'
definition for these applications. As evidence for this, note that
in \cite{Joyc} we will show that if $\bs X,\bs Y$ are d-manifolds
with corners, $Z$ is a manifold with corners, and $f:\bs X\ra Z$,
$g:\bs Y\ra Z$ are arbitrary smooth maps, then a fibre product $\bs
X\t_{f,Z,g}\bs Y$ exists in the 2-category $\dManc$ of d-manifolds
with corners. The fact that this works is crucially dependent on the
details of our definition of smooth map.

We begin in \S\ref{mc2} with definitions and properties of manifolds
with corners $X$, their boundaries $\pd X$, $k$-boundaries $\pd^kX$
and $k$-corners $C_k(X)\cong \pd^kX/S_k$. Section \ref{mc3} defines
and studies smooth maps $f:X\ra Y$ of manifolds with corners, and
\S\ref{mc4} explains two ways to encode how a smooth $f:X\ra Y$
relates $\pd^kX$ and $\pd^lY$ for $k,l\ge 0$. Sections
\ref{mc5}--\ref{mc7} discuss submersions, transversality and fibre
products of manifolds with corners, and orientations and orientation
conventions. The proofs of Theorems \ref{mc6thm1} and \ref{mc6thm2},
two of our main results on fibre products, are postponed to
\S\ref{mc8} and~\S\ref{mc9}.

The author would like to thank Franki Dillen for pointing out
reference~\cite{MaOu}.

\section{Manifolds with corners, and boundaries}
\label{mc2}

We define {\it manifolds without boundary}, {\it with boundary}, and
{\it with corners}.

\begin{dfn} Let $X$ be a paracompact Hausdorff topological space.
\begin{itemize}
\setlength{\itemsep}{0pt}
\setlength{\parsep}{0pt}
\item[(i)] An {\it $n$-dimensional chart on\/ $X$ without boundary\/} is
a pair $(U,\phi)$, where $U$ is an open subset in $\R^n$, and
$\phi:U\ra X$ is a homeomorphism with a nonempty open set $\phi(U)$
in~$X$.
\item[(ii)] An {\it $n$-dimensional chart on\/ $X$ with boundary\/} for
$n\ge 1$ is a pair $(U,\phi)$, where $U$ is an open subset in $\R^n$
or in $[0,\iy)\t\R^{n-1}$, and $\phi:U\ra X$ is a homeomorphism with
a nonempty open set $\phi(U)$.
\item[(iii)] An {\it $n$-dimensional chart on\/ $X$ with corners\/}
for $n\ge 1$ is a pair $(U,\phi)$, where $U$ is an open subset in
$\R^n_k=[0,\iy)^k\t\R^{n-k}$ for some $0\le k\le n$, and $\phi:U\ra
X$ is a homeomorphism with a nonempty open set~$\phi(U)$.
\end{itemize}
These are increasing order of generality, that is, (i) $\Rightarrow$
(ii) $\Rightarrow$ (iii).

For brevity we will use the notation $\R^n_k=[0,\iy)^k\t\R^{n-k}$,
following~\cite[\S 1.1]{Melr2}.

Let $A\subseteq\R^m$ and $B\subseteq\R^n$ and $\al:A\ra B$ be
continuous. We call $\al$ {\it smooth\/} if it extends to a smooth
map between open neighbourhoods of $A,B$, that is, if there exists
an open subset $A'$ of $\R^m$ with $A\subseteq A'$ and a smooth map
$\al':A'\ra\R^n$ with $\al'\vert_A\equiv\al$. If $A$ is open we can
take $A'=A$ and $\al'=\al$. When $m=n$ we call $\al:A\ra B$ a {\it
diffeomorphism\/} if it is a homeomorphism and $\al:A\ra B$,
$\al^{-1}:B\ra A$ are smooth.

Let $(U,\phi),(V,\psi)$ be $n$-dimensional charts on $X$, which may
be without boundary, or with boundary, or with corners. We call
$(U,\phi)$ and $(V,\psi)$ {\it compatible\/} if
$\psi^{-1}\ci\phi:\phi^{-1}\bigl(\phi(U)\cap\psi(V)\bigr)\ra
\psi^{-1}\bigl(\phi(U)\cap\psi(V)\bigr)$ is a diffeomorphism between
subsets of $\R^n$, in the sense above.

An $n$-{\it dimensional atlas\/} for $X$ {\it without boundary},
{\it with boundary}, or {\it with corners}, is a system
$\{(U^i,\phi^i):i\in I\}$ of pairwise compatible $n$-dimensional
charts on $X$ with $X=\bigcup_{i\in I}\phi^i(U^i)$, where the
$(U^i,\phi^i)$ are with boundary, or with corners, respectively. We
call such an atlas {\it maximal\/} if it is not a proper subset of
any other atlas. Any atlas $\{(U^i,\phi^i):i\in I\}$ is contained in
a unique maximal atlas of the same type, the set of all charts
$(U,\phi)$ of this type on $X$ which are compatible with
$(U^i,\phi^i)$ for all~$i\in I$.

An $n$-{\it dimensional manifold without boundary}, or {\it with
boundary}, or {\it with corners}, is a paracompact Hausdorff
topological space $X$ equipped with a maximal $n$-dimensional atlas
without boundary, or with boundary, or with corners, respectively.
Usually we refer to $X$ as the manifold, leaving the atlas implicit,
and by a {\it chart\/ $(U,\phi)$ on the manifold\/} $X$, we mean an
element of the maximal atlas. When we just say manifold, we will
usually mean a manifold with corners.
\label{mc2def1}
\end{dfn}

Here is some notation on (co)tangent spaces.

\begin{dfn} Let $X$ be an $n$-manifold with corners. A map $f:X\ra\R$ is
called {\it smooth\/} if whenever $(U,\phi)$ is a chart on the
manifold $X$ then $f\ci\phi:U\ra\R$ is a smooth map between subsets
of $\R^n,\R$, in the sense of Definition \ref{mc2def1}. Write
$C^\iy(X)$ for the $\R$-algebra of smooth functions~$f:X\ra\R$.

Following \cite[p.~4]{KoNo}, for each $x\in X$ define the {\it
tangent space\/} $T_xX$ by
\begin{align*}
T_xX=\bigl\{v:\,&\text{$v$ is a linear map $C^\iy(X)\ra\R$
satisfying}\\
&\text{$v(fg)=v(f)g(x)+f(x)v(g)$ for all $f,g\in C^\iy(X)$}\bigr\},
\end{align*}
and define the {\it cotangent space\/} $T_x^*X=(T_xX)^*$. Both are
vector spaces of dimension $n$. For each $x\in X$, we define the
{\it inward sector\/} $IS(T_xX)$ of vectors which `point into $X$',
as follows: let $(U,\phi)$ be a chart on $X$ with $U\subseteq\R^n_k$
open and $0\in U$ with $\phi(0)=x$. Then
$\d\phi\vert_0:T_0\R^n_k=\R^n\ra T_xX$ is an isomorphism. Set
$IS(T_xX)=\d\phi\vert_0(\R^n_k)\subseteq T_xX$. This is independent
of the choice of~$(U,\phi)$.
\label{mc2def2}
\end{dfn}

We now study the notion of {\it boundary\/} $\pd X$ for
$n$-manifolds $X$ with corners.

\begin{dfn} Let $U\subseteq\R^n_k$ be open. For each
$u=(u_1,\ldots,u_n)$ in $U$, define the {\it depth\/} $\depth_Uu$ of
$u$ in $U$ to be the number of $u_1,\ldots,u_k$ which are zero. That
is, $\depth_Uu$ is the number of boundary faces of $U$
containing~$u$.

Let $X$ be an $n$-manifold with corners. For $x\in X$, choose a
chart $(U,\phi)$ on the manifold $X$ with $\phi(u)=x$ for $u\in U$,
and define the {\it depth\/} $\depth_Xx$ of $x$ in $X$ by
$\depth_Xx=\depth_Uu$. This is independent of the choice of
$(U,\phi)$. For each $k=0,\ldots,n$, define the {\it depth\/ $k$
stratum\/} of $X$ to be
\begin{equation*}
S^k(X)=\bigl\{x\in X:\depth_Xx=k\bigr\}.
\end{equation*}\vspace{-20pt}
\label{mc2def3}
\end{dfn}

The proof of the next proposition is elementary.

\begin{prop} Let\/ $X$ be an $n$-manifold with corners. Then
\begin{itemize}
\setlength{\itemsep}{0pt}
\setlength{\parsep}{0pt}
\item[{\bf(a)}] $X=\coprod_{k=0}^nS^k(X),$ and\/
$\overline{S^k(X)}=\bigcup_{l=k}^n S^l(X);$
\item[{\bf(b)}] Each\/ $S^k(X)$ has the structure of an
$(n-k)$-manifold without boundary;
\item[{\bf(c)}] $X$ is a manifold without boundary if and only
if\/ $S^k(X)=\es$ for\/ $k>0;$
\item[{\bf(d)}] $X$ is a manifold with boundary if and only if\/
$S^k(X)=\es$ for\/ $k>1;$ and
\item[{\bf(e)}] If\/ $x\in S^k(X)$ then $IS(T_xX)$ in $T_xX$ is
isomorphic to $\R^n_k$ in $\R^n$. Also the intersection
$IS(T_xX)\cap -IS(T_xX)$ in $T_xX$ is~$T_xS^k(X)\cong\R^{n-k}$.
\end{itemize}
\label{mc2prop1}
\end{prop}

\begin{dfn} Let $X$ be a manifold with corners, and $x\in X$. A
{\it local boundary component\/ $\be$ of\/ $X$ at\/} $x$ is a local
choice of connected component of $S^1(X)$ near $x$. That is, for
each sufficiently small open neighbourhood $V$ of $x$ in $X$, $\be$
gives a choice of connected component $W$ of $V\cap S^1(X)$ with
$x\in\overline W$, and any two such choices $V,W$ and $V',W'$ must
be compatible in the sense that~$x\in\overline{(W\cap W')}$.
\label{mc2def4}
\end{dfn}

The meaning of local boundary components in coordinate charts is
easy to explain. Suppose $(U,\phi)$ is a chart on $X$ with
$\phi(u)=x$, where $U$ is an open set in $\R^n_k$, and write
$u=(u_1,\ldots,u_n)$. Then
\begin{equation*}
S^1(U)=\ts\coprod_{i=1}^k\bigl\{(x_1,\ldots,x_n)\in U:\text{$x_i=0$,
$x_j\ne 0$ for $j\ne i$, $j=1,\ldots,k$}\bigr\}.
\end{equation*}
If $u_i=0$ for some $i=1,\ldots,k$, then $\bigl\{(x_1,\ldots,x_n)\in
U:x_i=0$, $x_j\ne 0$ for $j\ne i$, $j=1,\ldots,k\bigr\}$ is a subset
of $S^1(U)$ whose closure contains $u$, and the intersection of this
subset with any sufficiently small open ball about $u$ is connected,
so this subset uniquely determines a local boundary component of $U$
at $u$, and hence a local boundary component of $X$ at $x$. Thus,
the local boundary components of $X$ at $x$ are in 1-1
correspondence with those $i=1,\ldots,k$ with $u_i=0$. But the
number of such $i$ is the depth $\depth_Xx$. Hence {\it there are
exactly $\depth_Xx$ distinct local boundary components\/ $\be$ of\/
$X$ at\/ $x$ for each\/}~$x\in X$.

\begin{dfn} Let $X$ be a manifold with corners. As a set,
define the {\it boundary\/}
\e
\pd X=\bigl\{(x,\be):\text{$x\in X$, $\be$ is a local boundary
component for $X$ at $x$}\bigr\}.
\label{mc2eq1}
\e
Define a map $i_X:\pd X\ra X$ by $i_X:(x,\be)\mapsto x$. Note that
$i_X$ {\it need not be injective}, as $\bmd{i_X^{-1}(x)}=\depth_Xx$
for all~$x\in X$.
\label{mc2def5}
\end{dfn}

If $(U,\phi)$ is a chart on the manifold $X$ with $U\subseteq\R^n_k$
open, then for each $i=1,\ldots,k$ we can define a chart
$(U_i,\phi_i)$ on $\pd X$ by
\e
\begin{split}
&U_i=\bigl\{(x_1,\ldots,x_{n-1})\in \R^{n-1}_{k-1}:
(x_1,\ldots,x_{i-1},0,x_i,\ldots,x_{n-1})\in
U\subseteq\R^n_k\bigr\},\\
&\phi_i:(x_1,\ldots,x_{n-1})\longmapsto\bigl(\phi
(x_1,\ldots,x_{i-1}, 0,x_i,\ldots,x_{n-1}),\phi_*(\{x_i=0\})\bigr).
\end{split}
\label{mc2eq2}
\e
All such charts on $\pd X$ are compatible, and induce a manifold
structure on $\pd X$. Thus as in Douady \cite[\S 6]{Doua} we may
prove:

\begin{prop} Let\/ $X$ be an $n$-manifold with corners. Then $\pd
X$ is naturally an $(n\!-\!1)$-manifold with corners for $n>0,$
and\/ $\pd X=\es$ if\/~$n=0$.
\label{mc2prop2}
\end{prop}

The map $i_X:\pd X\ra X$ is continuous, and we will see in
\S\ref{mc3} that it is smooth. By considering the local models
$\R^n_k$ for $X$, we see:

\begin{lem} As a map between topological spaces, $i_X:\pd X\ra X$ in
Definition {\rm\ref{mc2def5}} is continuous, finite (that is,
$i_X^{-1}(x)$ is finite for all\/ $x\in X$), and proper (that is,
if\/ $S\subseteq X$ is compact then $i_X^{-1}(S)\subseteq\pd X$ is
compact).
\label{mc2lem}
\end{lem}

As $\pd X$ is a manifold with corners we can iterate the boundary
construction to obtain $\pd X,\pd^2X,\ldots,\pd^nX$, with $\pd^kX$
an $(n-k)$-manifold with corners.

\begin{prop} Let\/ $X$ be an $n$-manifold with corners. Then for
$k=0,\ldots,n$ there are natural identifications
\e
\begin{split}
\pd^kX\cong\bigl\{(x,\be_1,\ldots,\be_k):\,&\text{$x\in X,$
$\be_1,\ldots,\be_k$ are distinct}\\
&\text{local boundary components for $X$ at $x$}\bigr\}.
\end{split}
\label{mc2eq3}
\e
\label{mc2prop3}
\end{prop}

\begin{proof} Consider first the case $k=2$. Points of $\pd^2X$ are
of the form $\bigl((x,\be_1),\ti\be_2\bigr)$, where $x\in X$, and
$\be_1$ is a local boundary component of $X$ at $x$, and $\ti\be_2$
is a local boundary component of $\pd X$ at $(x,\be_1)$. Suppose
$(U,\phi)$ is a chart for $X$ with $x=\phi(u)$ for some $u\in U$,
where $U$ is open in $\R^n_l$ for $l\ge 2$, and
$u=(u_1,\ldots,u_n)$, with $u_{i_1}=0$ and $\phi^{-1}(\be_1)$ the
local boundary component $x_{i_1}=0$ in $U$. Then \eq{mc2eq2} gives
a chart $(U_{i_1},\phi_{i_1})$ for $\pd X$ with
$\phi_{i_1}\bigl((u_1,\ldots,u_{i_1-1},u_{i_1+1},\ldots,
u_n)\bigr)=(x,\be_1)$. Thus $\phi_{i_1}^{-1}(\ti\be_2)$ is a local
boundary component for $U_{i_1}$, and so is of the form $x_j=0$ for
$j=1,\ldots,l-1$. Write $i_2=j$ if $j<i_1$ and $i_2=j+1$ if $j\ge
i_1$. Then $u_{i_2}=0$, as $u_{i_2}$ is the $j^{\rm th}$ coordinate
of~$(u_1,\ldots,u_{i_1-1},u_{i_1+1},\ldots,u_n)$.

Let $\be_2$ be the local boundary component $\phi_*(\{x_{i_2}=0\})$
of $X$ at $x$. Then $\be_2\ne \be_1$ as $i_2\ne i_1$. We have
constructed a 1-1 correspondence between local boundary components
$\ti\be_2$ of $\pd X$ at $(x,\be_1)$ and local boundary components
$\be_2$ of $X$ at $x$ with $\be_2\ne\be_1$. This 1-1 correspondence
is independent of the choice of chart $(U,\phi)$. Identifying
$\bigl((x,\be_1),\ti\be_2\bigr)$ with $(x,\be_1,\be_2)$ gives
\eq{mc2eq3} for~$k=2$.

We prove the general case by induction on $k$. The case $k=0$ is
trivial, $k=1$ is \eq{mc2eq1}, and $k=2$ we have proved above. For
the inductive step, having proved \eq{mc2eq3} for $k\le l<n$, we
show that at a point of $\pd^lX$ identified with
$(p,\be_1,\ldots,\be_l)$ under \eq{mc2eq3}, local boundary
components of $\pd^lX$ are in 1-1 correspondence with local boundary
components $\be_{l+1}$ of $X$ at $p$ distinct
from~$\be_1,\ldots,\be_l$.
\end{proof}

\begin{dfn} Write $S_k$ for the symmetric group on $k$ elements,
the group of bijections $\si:\{1,\ldots,k\}\ra\{1,\ldots,k\}$. From
\eq{mc2eq3} we see that $\pd^kX$ has a natural, free action of $S_k$
by permuting $\be_1,\ldots,\be_k$, given by
\begin{equation*}
\si:(x,\be_1,\ldots,\be_k)\longmapsto
(x,\be_{\si(1)},\ldots,\be_{\si(k)}).
\end{equation*}
Each $\si\in S_k$ acts on $\pd^kX$ as an isomorphism of
$(n-k)$-manifolds with corners (a diffeomorphism). Thus, if $G$ is a
subgroup of $S_k$, then the quotient $(\pd^kX)/G$ is also an
$(n-k)$-manifold with corners.

In particular, taking $G=S_k$, we define the $k$-{\it corners\/}
$C_k(X)$ of $X$ to be
\e
\begin{aligned}
C_k(X)=\bigl\{(x,&\{\be_1,\ldots,\be_k\}):\text{$x\in X,$
$\be_1,\ldots,\be_k$ are distinct}\\
&\text{local boundary components for $X$ at $x$}\bigr\}\cong
\pd^kX/S_k,
\end{aligned}
\label{mc2eq4}
\e
an $(n-k)$-manifold with corners. Note that $\be_1,\ldots,\be_k$ are
unordered in \eq{mc2eq4}, but ordered in \eq{mc2eq3}. We have
isomorphisms $C_0(X)\cong X$ and~$C_1(X)\cong\pd X$.
\label{mc2def6}
\end{dfn}

\begin{rem} We review how our definitions so far relate to those in
use by other authors. For manifolds without or with boundary, all
definitions the author has found (see Kobayashi and Nomizu
\cite{KoNo} or Lang \cite{Lang}, for instance) are equivalent to
Definition \ref{mc2def1}. However, for manifolds with corners, there
are four main {\it inequivalent\/} definitions. Our terminology for
(a),(b),(d) follows J\"anich \cite[\S 1.1]{Jani}, and for (c)
follows Monthubert~\cite[\S 2.2]{Mont}.
\begin{itemize}
\setlength{\itemsep}{0pt}
\setlength{\parsep}{0pt}
\item[(a)] {\it Manifolds with corners\/} are as in
Definition~\ref{mc2def1}.
\item[(b)] A manifold with corners $X$ is called a {\it manifold
with faces\/} if each $x\in X$ lies in the image of
$\depth_Xx=\bmd{i_X^{-1}(x)}$ different connected components of
$\pd X$ under~$i_X:\pd X\ra X$.
\item[(c)] A manifold with corners $X$ is called a {\it manifold
with embedded corners\/} if there exists a decomposition $\pd
X=\pd_1X\amalg \pd_2 X\amalg \cdots\amalg\pd_NX$ for finite $N\ge 0$
with $\pd_iX$ open and closed in $\pd X$, such that
$i_X\vert_{\pd_iX}:\pd_iX\ra X$ is injective for $i=1,\ldots,N$. We
allow~$\pd_iX=\es$.
\item[(d)] For each integer $N\ge 0$, an $\an{N}$-{\it manifold\/}
is a manifold with corners $X$ together with a given decomposition
$\pd X=\pd_1X\amalg\cdots\amalg\pd_NX$ with $\pd_iX$ open and closed
in $\pd X$, such that $i_X\vert_{\pd_iX}:\pd_iX\ra X$ is injective
for $i=1,\ldots,N$. We allow $\pd_iX=\es$. Note that $N$ has no
relation to $\dim X$. A $\an{0}$-manifold is a manifold without
boundary, and a $\an{1}$-manifold is a manifold with boundary.
\end{itemize}

Note that (c) implies (b) implies (a), and (d) becomes (c) after
forgetting the decomposition $\pd X=\pd_1X\amalg\cdots\amalg
\pd_NX$. An example satisfying (a) but not (b)--(d) is the {\it
teardrop\/} $T=\bigl\{(x,y)\in\R^2:x\ge 0$, $y^2\le x^2-x^4\bigr\}$,
shown in Figure \ref{fig1}. For compact $X$ (b) also implies (c),
but one can find pathological examples of noncompact $X$ which
satisfy (b), but not (c) or (d). Cerf \cite{Cerf}, Douady
\cite{Doua}, Margalef-Roig and Outerelo Dominguez \cite{MaOu}, and
others define manifolds with corners as in (a). Melrose
\cite{Melr1,Melr2} and authors who follow him define manifolds with
corners as in (c); Melrose \cite[\S 1.6]{Melr2} calls manifolds with
corners in sense (a) {\it t-manifolds}. J\"anich \cite[\S 1.1]{Jani}
defines manifolds with corners in senses (a),(b) and (d), but is
primarily interested in (d). Laures \cite{Laur} also works with
$\an{N}$-manifolds, in sense~(d).
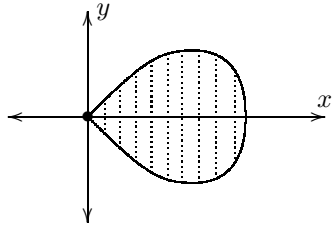
\begin{figure}[htb]
\begin{xy}
,(-1.5,0)*{}
,<6cm,-1.5cm>;<6.7cm,-1.5cm>:
,(3,.3)*{x}
,(-1.2,2)*{y}
,(-1.5,0)*{\bullet}
,(-1.5,0); (1.5,0) **\crv{(-.5,1)&(.1,1.4)&(1.5,1.2)}
?(.06)="a"
?(.12)="b"
?(.2)="c"
?(.29)="d"
?(.4)="e"
?(.5)="f"
?(.6)="g"
?(.7)="h"
?(.83)="i"
,(-1.5,0); (1.5,0) **\crv{(-.5,-1)&(.1,-1.4)&(1.5,-1.2)}
?(.06)="j"
?(.12)="k"
?(.2)="l"
?(.29)="m"
?(.4)="n"
?(.5)="o"
?(.6)="p"
?(.7)="q"
?(.83)="r"
,"a";"j"**@{.}
,"b";"k"**@{.}
,"c";"l"**@{.}
,"d";"m"**@{.}
,"e";"n"**@{.}
,"f";"o"**@{.}
,"g";"p"**@{.}
,"h";"q"**@{.}
,"i";"r"**@{.}
\ar (-1.5,0);(3,0)
\ar (-1.5,0);(-3,0)
\ar (-1.5,0);(-1.5,2)
\ar (-1.5,0);(-1.5,-2)
\end{xy}
\caption{The teardrop, a 2-manifold with corners of type (a)}
\label{fig1}
\end{figure}

The {\it boundary\/} $\pd X$ of a manifold with corners $X$ is also
defined in different ways in the literature. In our picture, $\pd X$
is a manifold with corners, with an immersion $i_X:\pd X\ra X$ which
is not necessarily injective, so that $\pd X$ may not be a subset of
$X$. This follows Douady \cite[\S 6]{Doua}, who defines $\pd^kX$ (in
his notation) to be equivalent to our $C_k(X)$ in \eq{mc2eq4}, so
that his $\pd^1X$ agrees with our $\pd X$. All the other authors
cited define $\pd X$ to be $i_X(\pd X)$ in our notation, so that
$\pd X$ is a subset of $X$, but is not necessarily a manifold with
corners. But in (c),(d) above, the $\pd_iX$ are both subsets of $X$
and manifolds with corners.
\label{mc2rem}
\end{rem}

If $X,Y$ are manifolds with corners of dimensions $m,n$, there is a
natural way to make the product $X\t Y$ into a manifold with
corners, of dimension $m+n$. The following result on boundary and
$k$-corners of $X\t Y$ is easy to prove by considering local models
$\R^{m+n}_{a+b}\cong \R^m_a\t\R^n_b$ for~$X\t Y$.

\begin{prop} Let\/ $X,Y$ be manifolds with corners. Then there are
natural isomorphisms of manifolds with corners
\ea
\pd(X\t Y)&\cong (\pd X\t Y)\amalg (X\t\pd Y),
\label{mc2eq5}\\
C_k(X\t Y)&\cong \ts\coprod_{i,j\ge 0,\; i+j=k}C_i(X)\t C_j(Y).
\label{mc2eq6}
\ea
\label{mc2prop4}
\end{prop}

Note that \eq{mc2eq5} and \eq{mc2eq6} imply that
\ea
\pd^k(X\t Y)&\cong\ts \coprod_{i=1}^k\binom{k}{i}\,
\pd^iX\t\pd^{k-i}Y,
\label{mc2eq7}\\
\ts\coprod_{k=0}^{\dim X\t Y}C_k(X\t Y)&\cong
\ts\bigl[\coprod_{i=0}^{\dim X}C_i(X)
\bigr]\t\bigl[\coprod_{j=0}^{\dim Y}C_j(Y)\bigr].
\label{mc2eq8}
\ea
We will see in \S\ref{mc4} that if $X$ is a manifold with corners
then we can make $\coprod_{i=0}^{\dim X}C_i(X)$ behave functorially
under smooth maps.

The map $X\mapsto C_k(X)$ commutes with boundaries. The proof is
again an easy exercise by considering local models $\R^m_a$ for~$X$.

\begin{prop} Let\/ $X$ be a manifold with corners and\/ $k\ge 0$. Then
there are natural identifications, with the first a diffeomorphism:
\e
\begin{split}
\pd\bigl(C_k(X)\bigr)\cong C_k(\pd X)\cong
\bigl\{(x,\be_1,\{\be_2,\ldots,\be_{k+1}\}):\text{$x\in X,$
$\be_1,\ldots,\be_{k+1}$}&\\
\text{are distinct local boundary components for\/ $X$ at\/
$x$}&\bigr\}.
\end{split}
\label{mc2eq9}
\e
\label{mc2prop5}
\end{prop}

The next definition will be used in defining smooth maps
in~\S\ref{mc3}.

\begin{dfn} Let $X$ be a manifold with corners, and $(x,\be)\in\pd
X$. A {\it boundary defining function for\/ $X$ at\/} $(x,\be)$ is a
pair $(V,b)$, where $V$ is an open neighbourhood of $x$ in $X$ and
$b:V\ra[0,\iy)$ is a map, such that $b:V\ra\R$ is smooth in the
sense of Definition \ref{mc2def2}, and $\d b\vert_x:T_xV\ra
T_0[0,\iy)$ is nonzero, and there exists an open neighbourhood $\ti
V$ of $(x,\be)$ in $i_X^{-1}(V)\subseteq\pd X$, with $b\ci
i_X\vert_{\ti V}\equiv 0$, and $i_X\vert_{\ti V}:\ti
V\longra\bigl\{x'\in V:b(x')=0\bigr\}$ is a homeomorphism between
$\ti V$ and an open subset of~$\bigl\{x'\in V:b(x')=0\bigr\}$.

Thus the boundary $\pd X$ is defined near $(x,\be)$ by the equation
$b=0$ in $X$ near $x$. Using the ideas on fibre products of
manifolds with corners in \S\ref{mc6}, one can say more: $\pd X$
near $(x,\be)$ is naturally isomorphic to the fibre product of
manifolds with corners $V\t_{b,[0,\iy),i}\{0\}$ near $(x,0)$, where
$i:\{0\}\ra[0,\iy)$ is the inclusion, and the fibre product exists
near~$(x,0)$.
\label{mc2def7}
\end{dfn}

Here are some properties of such $(V,b)$. The proofs are elementary.

\begin{prop} Let\/ $X$ be an $n$-manifold with corners,
and\/~$(x,\be)\in\pd X$.
\begin{itemize}
\setlength{\itemsep}{0pt}
\setlength{\parsep}{0pt}
\item[{\bf(a)}] There exists a boundary defining function\/
$(V,b)$ for\/ $X$ at\/ $(x,\be)$.
\item[{\bf(b)}] Let\/ $(V,b)$ and\/ $(V',b')$ be boundary defining
functions for $X$ at\/ $(x,\be)$. Then there exists an open
neighbourhood\/ $V''$ of\/ $x$ in $V\cap V'$ and a smooth function
$g:V''\ra(0,\iy)\subset\R$ such that\/ $b'\vert_{V''}\equiv
b\vert_{V''}\cdot g$.
\item[{\bf(c)}] Let\/ $(V,b)$ be a boundary defining function for
$X$ at\/ $(x,\be)$. Then there exists a chart\/ $(U,\phi)$ on
the manifold\/ $X,$ such that\/ $U$ is open in\/ $\R^n_k$ for\/
$0<k\le n$ and\/ $0\in U$ with\/ $\phi(0)=x,$ and\/ $\be$ is the
image of the local boundary component\/ $x_1=0$ of\/ $U$ at\/
$0,$ and\/ $\phi(U)\subseteq V,$ and\/ $b\ci\phi\equiv
x_1:U\ra[0,\iy)$.
\item[{\bf(d)}] Let\/ $(U,\phi)$ be a chart on the
manifold\/ $X,$ such that\/ $U$ is open in\/ $\R^n_k$ and\/ $u\in U$
with\/ $\phi(u)=x,$ and\/ $\be$ is the image of the local boundary
component $x_i=0$ of\/ $U$ at\/ $u$ for $i\le k$. Then
$\bigl(\phi(U),x_i\ci\phi^{-1}\bigr)$ is a boundary defining
function for $X$ at\/~$(x,\be)$.
\end{itemize}
\label{mc2prop6}
\end{prop}

\section{Smooth maps of manifolds with corners}
\label{mc3}

Here is our definition of {\it smooth maps\/} $f:X\ra Y$ of
manifolds with corners~$X,Y$.

\begin{dfn} Let $X,Y$ be manifolds with corners of dimensions
$m,n$. A continuous map $f:X\ra Y$ is called {\it weakly smooth\/}
if whenever $(U,\phi),(V,\psi)$ are charts on the manifolds $X,Y$
then
\begin{equation*}
\psi^{-1}\ci f\ci\phi:(f\ci\phi)^{-1}(\psi(V))\longra V
\end{equation*}
is a smooth map from $(f\ci\phi)^{-1}(\psi(V))\subset\R^m$ to
$V\subset\R^n$, where smooth maps between subsets of $\R^m,\R^n$ are
defined in Definition~\ref{mc2def1}.

A weakly smooth map $f:X\ra Y$ is called {\it smooth\/} if it
satisfies the following additional condition over $\pd X,\pd Y$.
Suppose $x\in X$ with $f(x)=y\in Y$, and $\be$ is a local boundary
component of $Y$ at $y$. Let $(V,b)$ be a boundary defining function
for $Y$ at $(y,\be)$. Then $f^{-1}(V)$ is an open neighbourhood of
$x$ in $X$, and $b\ci f:f^{-1}(V)\ra[0,\iy)$ is a weakly smooth map.
We require that either $b\ci f\equiv 0$ on an open neighbourhood of
$x$ in $f^{-1}(V)$, or $(f^{-1}(V),b\ci f)$ is a boundary defining
function for $X$ at $(x,\ti\be)$, for some unique local boundary
component $\ti\be$ of $X$ at~$x$.
\label{mc3def1}
\end{dfn}

We also define five special classes of smooth maps:

\begin{dfn} Let $X,Y$ be manifolds with corners of dimensions $m,n$,
and $f:X\ra Y$ a weakly smooth map. If $x\in X$ with $f(x)=y$ then
in the usual way there is an induced linear map on tangent spaces
$\d f\vert_x:T_xX\ra T_yY$. In the notation of Definition
\ref{mc2def2}, $\d f\vert_x:T_xX\ra T_yY$ maps $IS(T_xX)\ra
IS(T_yY)$, that is, $\d f\vert_x$ maps inward-pointing vectors to
inward-pointing vectors.

Let $x\in S^k(X)$ and $y\in S^l(Y)$. Then the inclusion
$T_x(S^k(X))\subseteq IS(T_xX)\subseteq T_xX$ is modelled on
$\{0\}\t\R^{n-k}\subseteq [0,\iy)^k\t\R^{n-k}\subseteq \R^n$. Hence
$T_x(S^k(X))=IS(T_xX)\cap -IS(T_xX)$, and similarly
$T_y(S^l(Y))=IS(T_yY)\cap -IS(T_yY)$. Since $\d f\vert_x$ maps
$IS(T_xX)\ra IS(T_yY)$ it maps $IS(T_xX)\cap -IS(T_xX)\ra IS(T_yY)
\cap -IS(T_yY)$, that is, $\d f\vert_x$ maps $T_x(S^k(X))\ra
T_y(S^l(Y))$. Hence there is an induced linear map
\e
(\d f\vert_x)_*:T_xX/T_x(S^k(X))\longra T_yY/T_y(S^l(Y)).
\label{mc3eq1}
\e

Now let $f:X\ra Y$ be a smooth map.
\begin{itemize}
\setlength{\itemsep}{0pt}
\setlength{\parsep}{0pt}
\item[(i)] We call $f$ a {\it diffeomorphism\/} if $f$ has a smooth
inverse $f^{-1}:Y\ra X$.
\item[(ii)] We call $f$ an {\it immersion\/} if $\d f\vert_x:T_xX\ra
T_{f(x)}Y$ is injective for all~$x\in X$.
\item[(iii)] We call $f$ an {\it embedding\/} if it is an injective
immersion.
\item[(iv)] We call $f$ a {\it submersion\/} if $\d f\vert_x:T_xX\ra
T_{f(x)}Y$ and $\d f\vert_x:T_x(S^k(X))\ra T_{f(x)}(S^l(Y))$ are
surjective for all $x\in X$, where $x\in S^k(X)$, $f(x)\in S^l(Y)$.
\item[(v)] We call $f$ {\it boundary-submersive}, or {\it
b-submersive}, if $(\d f\vert_x)_*$ in \eq{mc3eq1} is surjective
for all $x\in X$. Note that $\d f\vert_x$ surjective implies
$(\d f\vert_x)_*$ surjective, so submersions are automatically
b-submersive.
\end{itemize}
\label{mc3def2}
\end{dfn}

Here is how Definition \ref{mc3def1} relates to other definitions in
the literature:

\begin{rem} Weakly smooth maps $f:X\ra Y$ are just the obvious
generalization of the usual definition \cite[\S I.1]{KoNo} of smooth
maps for manifolds without boundary. If $\pd Y=\es$ the additional
condition in Definition \ref{mc3def1} is vacuous, and weakly smooth
maps are smooth. Note that the definition of smooth maps $f:X\ra\R$
in Definition \ref{mc2def2} is equivalent to Definition
\ref{mc3def1} when~$Y=\R$.

Our definition of smooth maps between manifolds with corners is not
equivalent to any other definition that the author has found in the
literature, though it is related. Most authors, such as Cerf
\cite[\S I.1.2]{Cerf}, define smooth maps of manifolds with corners
to be weakly smooth maps, in our notation. But there are also two
more complex definitions. Firstly, Monthubert \cite[Def.~2.8]{Mont}
defines {\it morphisms of manifolds with corners\/} $f:X\ra Y$. One
can show that these are equivalent to {\it b-submersive smooth
maps}, in our notation. We prefer our definition, as b-submersive
smooth maps do not have all the properties we want. In particular,
Theorem \ref{mc3thm}(iv),(vi) below fail for b-submersive smooth
maps.

Secondly, Melrose \cite[\S 1.12]{Melr2} defines {\it b-maps\/}
between manifolds with corners. Let $f:X\ra Y$ be a weakly smooth
map. We call $f$ a {\it b-map\/} if the following holds. Let $x\in
X$ with $f(x)=y$, and let the local boundary components of $X$ at
$x$ be $\ti\be_1,\ldots,\ti\be_k$, and of $Y$ at $y$ be
$\be_1,\ldots,\be_l$. Suppose $(\ti V_i,\ti b_i)$ is a boundary
defining function for $X$ at $(x,\ti\be_i)$, $i=1,\ldots,k$, and
$(V_j,b_j)$ a boundary defining function for $Y$ at $(y,\be_j)$,
$j=1,\ldots,l$. Then for all $j=1,\ldots,l$ either $b_j\ci f$ should
be zero near $x$ in $X$, or there should exist
$e_{1j},\ldots,e_{kj}\in\N$ such that $b_j\ci f\equiv
G_j\cdot\prod_{i=1}^k\ti b{}_i^{e_{ij}}$ near $x$ in $X$ for smooth
$G_j>0$. Thus, a smooth map in the sense of Definition \ref{mc3def1}
is exactly a b-map $f:X\ra Y$ such that for all such $x,y$ and
$j=1,\ldots,l$, one of $e_{1j},\ldots,e_{kj}$ is 1 and the rest are
zero. So our smooth maps are a special class of Melrose's b-maps.
\label{mc3rem}
\end{rem}

Here are some properties of smooth maps. The proofs are elementary.

\begin{thm} Let\/ $W,X,Y,Z$ be manifolds with corners.
\begin{itemize}
\setlength{\itemsep}{0pt}
\setlength{\parsep}{0pt}
\item[{\rm(i)}] If\/ $f:X\ra Y$ and\/ $g:Y\ra Z$ are smooth then\/
$g\ci f:X\ra Z$ is smooth.
\item[{\rm(ii)}] The identity map\/ $\id_X:X\ra X$ is smooth.
\item[{\rm(iii)}] Diffeomorphisms\/ $f:X\ra Y$ are equivalent to
isomorphisms of smooth manifolds, that is, to homeomorphisms of
topological spaces $f:X\ra Y$ which identify the maximal atlases on
$X$ and~$Y$.
\item[{\rm(iv)}] The map $i_X:\pd X\ra X$ in Definition\/
{\rm\ref{mc2def5}} is a smooth immersion.
\item[{\rm(v)}] If\/ $f:W\ra Y$ and\/ $g:X\ra Z$ are smooth, the
\begin{bfseries}product\end{bfseries} $f\t g:W\t X\ra Y\t Z$ given
by $(f\t g)(w,x)=\bigl(f(w),g(x)\bigr)$ is smooth.
\item[{\rm(vi)}] If\/ $f:X\ra Y$ and\/ $g:X\ra Z$ are smooth, the
\begin{bfseries}direct product\end{bfseries} $(f,g):X\ra Y\t Z$ given
by $(f,g)(x)=\bigl(f(x),g(x)\bigr)$ is smooth.
\item[{\rm(vii)}] Regarding the empty set\/ $\es$ as a manifold and
the point\/ $\{0\}$ as a $0$-manifold, the unique maps\/
$\es:\es\ra X$ and\/ $\pi:X\ra\{0\}$ are smooth.
\end{itemize}
\label{mc3thm}
\end{thm}

Theorem \ref{mc3thm}(i),(ii) show that manifolds with corners form a
{\it category\/}, which we write $\Manc$, with objects manifolds
with corners $X$ and morphisms smooth maps $f:X\ra Y$. We write
$\Manb$ for the full subcategory of $\Manc$ whose objects are
manifolds with boundary, and $\Man$ for the full subcategory of
$\Manc$ whose objects are manifolds without boundary, so that $\Man
\subset\Manb\subset\Manc$. Theorem \ref{mc3thm}(v),(vi) have a
category-theoretic interpretation in terms of products in $\Manc$,
and (vii) says that $\es$ is an {\it initial object\/} in $\Manc$,
and $\{0\}$ is a {\it terminal object\/} in $\Manc$. Here are some
examples.

\begin{ex}{\bf(a)} If $X$ is a manifold with corners, the diagonal
map $\De_X:X\ra X\t X$, $\De_X:x\mapsto(x,x)$, is a smooth
embedding. This follows from Theorem \ref{mc3thm}(ii),(vi), as
$\De_X=(\id_X,\id_X)$. If $\pd X\ne\es$ then $\De_X$ is not
b-submersive, so it is not a morphism of manifolds in the sense of
Monthubert~\cite[Def.~2.8]{Mont}.
\smallskip

\noindent{\bf(b)} If $X,Y$ are manifolds with corners then the
projection $\pi_X:X\t Y\ra X$ is a smooth submersion. This follows
from Theorem \ref{mc3thm}(ii),(v),(vii), by identifying $\pi_X:X\t
Y\ra X$ with $\id_X\t\pi:X\t Y\ra X\t\{0\}$.
\smallskip

\noindent{\bf(c)} The inclusion $i:[0,\iy)\ra\R$ is smooth, but it
is not a submersion, since at $0\in[0,\iy)$ the map $\d
i\vert_0:T_0S^0\bigl([0,\iy)\bigr)\ra T_0S^1(\R)$ is not surjective.
\smallskip

\noindent{\bf(d)} The map $f:\R\ra[0,\iy)$, $f(x)=x^2$ is weakly
smooth but {\it not\/} smooth, as the additional condition in
Definition \ref{mc3def1} fails at~$x=0$.
\smallskip

\noindent{\bf(e)} The map $f:[0,\iy)^2\ra[0,\iy)$, $f(x,y)=x+y$ is
weakly smooth but {\it not\/} smooth, as Definition \ref{mc3def1}
fails at~$(x,y)=(0,0)$.
\smallskip

\noindent{\bf(f)} The map $f:[0,\iy)^2\ra[0,\iy)$, $f(x,y)=xy$ is
weakly smooth but {\it not\/} smooth, as Definition \ref{mc3def1}
fails at $(x,y)=(0,0)$. However, $f$ is a b-map in the sense of
Melrose \cite[\S 1.12]{Melr2}, with~$e_{11}=e_{21}=1$.
\label{mc3ex}
\end{ex}

\section{Describing how smooth maps act on corners}
\label{mc4}

If $f:X\ra Y$ is a smooth map of manifolds with corners, then $f$
may relate $\pd^kX$ to $\pd^lY$ for $k,l\ge 0$ in complicated ways.
We now explain two different ways to describe these relations. The
first involves a decomposition $X\t_Y\pd Y=\Xi^f_+\amalg\Xi^f_-$ and
maps $\xi^f_+:\Xi^f_+\ra X$ and $\xi^f_-:\Xi^f_-\ra\pd X$. This will
be important in \cite{Joyc} when we define {\it d-manifolds with
corners\/} and {\it d-orbifolds with corners}, which are `derived'
versions of manifolds and orbifolds with corners. To make this
generalization we find it helpful to replace a manifold with corners
$X$ by the triple $(X,\pd X,i_X)$, so we need to characterize smooth
maps $f:X\ra Y$ in terms of the triples~$(X,\pd X,i_X),(Y,\pd
Y,i_Y)$.

\begin{dfn} Let $X,Y$ be manifolds with corners, and $f:X\ra Y$ a
smooth map. Consider the smooth maps $f:X\ra Y$ and $i_Y:\pd Y\ra Y$
as continuous maps of topological spaces. Then we may form the fibre
product of topological spaces $X\t_Y\pd Y=X\t_{f,Y,i_Y}\pd Y$, given
explicitly by
\begin{equation*}
X\t_Y\pd Y=\bigl\{\bigl(x,(y,\be)\bigr)\in X\t\pd
Y:f(x)=y=i_Y(y,\be)\bigr\}.
\end{equation*}
This is a closed subspace of the topological space $X\t\pd Y$, since
$X,\pd Y$ are Hausdorff, and so it is a topological space with the
subspace topology.

By Definition \ref{mc3def1}, for each $\bigl(x,(y,\be)\bigr)\in
X\t_Y\pd Y$, if $(V,b)$ is a boundary defining function for $Y$ at
$(y,\be)$, then either $b\ci f\equiv 0$ near $x$, or
$(f^{-1}(V),b\ci f)$ is a boundary defining function for $X$ at some
$(x,\ti\be)$. Define subsets $\Xi^f_+,\Xi^f_-$ of $X\t_Y\pd Y$ by
$\bigl(x,(y,\be)\bigr)\in\Xi^f_+$ if $b\ci f\equiv 0$ near $x$, and
$\bigl(x,(y,\be)\bigr)\in\Xi^f_-$ otherwise. Define maps
$\xi^f_+:\Xi^f_+\ra X$ by $\xi^f_+\bigl(x,(y,\be)\bigr)=x$ and
$\xi^f_-:\Xi^f_-\ra\pd X$ by $\xi^f_-\bigl(x,(y,\be)\bigr)
=(x,\ti\be)$, for $(x,\ti\be)$ as above. It is easy to show that
$\Xi^f_\pm,\xi^f_\pm$ can also be defined solely in terms of
$\bigl(x,(y,\be)\bigr)$ and $\d f\vert_x$, and so they are
independent of the choice of $(V,b)$, and are well-defined.
\label{mc4def1}
\end{dfn}

Here are some properties of these $\Xi^f_\pm,\xi^f_\pm$. A
continuous map $g:X\ra Y$ is a {\it finite covering map\/} if every
$y\in Y$ has an open neighbourhood $U$ such that $g^{-1}(U)$ is
homeomorphic to $U\t T$ for some finite set $T$ with the discrete
topology, and $g:\xi_f^{-1}(U)\ra U$ corresponds to the
projection~$U\t S\ra U$.

\begin{prop} Let\/ $f:X\ra Y$ be a smooth map of manifolds with
corners, and\/ $\Xi^f_\pm,\xi^f_\pm$ be as in Definition
{\rm\ref{mc4def1},} and set\/ $n=\dim X$. Then
\begin{itemize}
\setlength{\itemsep}{0pt}
\setlength{\parsep}{0pt}
\item[{\rm(a)}] $\Xi^f_\pm$ are open and closed subsets of\/
$X\t_Y\pd Y,$ with\/~$X\t_Y\pd Y=\Xi^f_+\amalg\Xi^f_-$.
\item[{\rm(b)}] $\xi^f_+:\Xi^f_+\ra X$ and\/ $\xi^f_-:\Xi^f_-\ra\pd
X$ are proper, finite covering maps of topological spaces, with\/
$\smash{\xi^f_+\equiv\pi_X\vert_{\Xi^f_+}}$ and\/~$\smash{i_X\ci
\xi^f_-\equiv\pi_X\vert_{\Xi^f_-}}$.
\item[{\rm(c)}] Part\/ {\rm(b)} implies there is a unique way
to make\/ $\smash{\Xi^f_+}$ into an $n$-manifold with corners and\/
$\Xi^f_-$ into an $(n\!-\!1)$-manifold with corners so that\/
$\xi^f_\pm$ are local diffeomorphisms, and so covering maps of
manifolds. Then the projections $\pi_X:\Xi^f_\pm\ra X$ and\/
$\pi_{\pd Y}:\Xi^f_\pm\ra\pd Y$ are smooth maps.
\end{itemize}
\label{mc4prop1}
\end{prop}

\begin{proof} For (a), clearly $X\t_Y\pd Y=\Xi^f_+\amalg\Xi^f_-$.
Let $\bigl(x,(y,\be)\bigr)\in X\t_Y\pd Y$, and $(V,b)$ be a boundary
defining function for $Y$ at $(y,\be)$. Then $(V,b)$ is also a
boundary defining function for $Y$ at any $(y',\be')$ sufficiently
close to $(y,\be)$ in $\pd Y$. Hence if $\bigl(x',(y',\be')\bigr)$
is sufficiently close to $\bigl(x,(y,\be)\bigr)\in X\t_Y\pd Y$ then
$(V,b)$ is also a boundary defining function for $Y$ at $(y',\be')$.
We have $\bigl(x,(y,\be)\bigr)\in\Xi^f_+$ if $b\ci f\equiv 0$ near
$x$. Fixing $(V,b)$ this is an open condition in $x$, so $\Xi^f_+$
is open in $X\t_Y\pd Y$, and thus $\Xi^f_-=(X\t_Y\pd Y)\sm\Xi^f_+$
is closed in $X\t_Y\pd Y$. Similarly, $\bigl(x,(y,\be)\bigr)
\in\Xi^f_-$ if $(f^{-1}(V),b\ci f)$ is a boundary defining function
for $X$ at some $(x,\ti\be)$. Fixing $(V,b)$ this is an open
condition in $(x,\ti\be)$, so $\Xi^f_-$ is open, and
$\Xi^f_+=(X\t_Y\pd Y)\sm\Xi^f_-$ is closed, proving~(a).

For (b), the identities $\smash{\xi^f_+\equiv\pi_X\vert_{\Xi^f_+}}$
and $\smash{i_X\ci \xi^f_-\equiv\pi_X\vert_{\Xi^f_-}}$ are
immediate. First consider $\xi^f_+:\Xi^f_+\ra X$. Since $i_Y:\pd
Y\ra Y$ is proper and finite by Lemma \ref{mc2lem}, $\pi_X:X\t_Y\pd
Y\ra X$ is proper and finite by properties of topological fibre
products, and so $\xi^f_+=\pi_X\vert_{\smash{\Xi^f_+}}:\Xi^f_+\ra X$
is proper and finite as $\Xi^f_+$ is closed in $X\t_Y\pd Y$. To see
that $\xi^f_+$ is a covering map, note that it is a local
homeomorphism, since as above, given
$\bigl(x,(y,\be)\bigr)\in\Xi^f_+$, if $x'$ is close to $x$ in $X$
then setting $y'=f(x')$, $(V,b)$ is a boundary defining function for
$Y$ at $(y',\be')$ for some unique local boundary component $\be'$
of $Y$ at $\be'$, and then $\bigl(x',(y',\be')\bigr)\in\Xi^f_+$ with
$\xi^f_+\bigl(x',(y',\be')\bigr)=x'$. We have constructed a local
inverse $x'\mapsto\bigl(x',(y',\be')\bigr)$ for $\xi^f_+$ which is
clearly continuous, so $\xi^f_+$ is a local homeomorphism, and thus
a finite covering map, as it is finite.

Next consider $\xi^f_-:\Xi^f_-\ra\pd X$. As above, given
$\bigl(x,(y,\be)\bigr)\in\Xi^f_-$ we may fix a boundary defining
function $(V,b)$ for $Y$ at $(y,\be)$, and then for
$\bigl(x',(y',\be')\bigr)$ near $\bigl(x,(y,\be)\bigr)$ in $\Xi^f_-$
we have $\xi^f_-\bigl(x',(y',\be')\bigr)=(x',\ti\be')$, where
$\ti\be'$ is the unique local boundary component of $X$ at $x'$ such
that $(f^{-1}(V),b\ci f)$ is a boundary defining function for $X$ at
$(x',\ti\be')$. Therefore $\xi^f_-$ is continuous, as $\ti\be'$
depends continuously on $x'$. As above $\pi_X:X\t_Y\pd Y\ra X$ is
proper and finite, so
$i_X\ci\xi^f_-=\pi_X\vert_{\smash{\Xi^f_-}}:\Xi^f_- \ra X$ is proper
and finite as $\Xi^f_-$ is closed, and hence $\xi^f_-$ is proper and
finite. We show $\xi^f_-$ is a finite covering map by constructing a
local inverse $(x',\ti\be')\mapsto\bigl(x',(y',\be')\bigr)$ as for
$\xi^f_+$. This proves~(b).

For (c), $\pi_X:\Xi^f_+\ra X$ is $\xi^f_+:\Xi^f_+\ra X$ and
$\pi_X:\Xi^f_-\ra X$ is $i_X\ci\xi^f_-:\Xi^f_-\ra X$, so
$\pi_X:\Xi^f_\pm\ra X$ are smooth as $\xi^f_\pm$ are covering maps
of manifolds, and so smooth. To see $\pi_{\pd Y}:\Xi^f_\pm\ra\pd Y$
are smooth, note that $i_Y\ci\pi_{\pd Y}\equiv f\ci\pi_X$ as maps
$\Xi^f_\pm\ra Y$, so $i_Y\ci\pi_{\pd Y}:\Xi^f_\pm\ra Y$ is smooth,
and it follows that $\pi_{\pd Y}:\Xi^f_\pm\ra\pd Y$ is smooth as
$\pi_{\pd Y}$ is continuous and $i_Y:\pd Y\ra Y$ is an immersion.
\end{proof}

Using $\Xi^f_\pm,\xi^f_\pm$ we can define a decomposition~$\pd
X=\pd_+^fX\amalg\pd_-^fX$.

\begin{prop} Let\/ $f:X\ra Y$ be a smooth map of manifolds with
corners. Define $\pd_-^fX=\xi_-^f(\Xi_-^f)$ and\/ $\pd_+^fX=\pd
X\sm\xi_-^f(\Xi_-^f),$ so that\/ $\pd X=\pd_+^fX\amalg\pd_-^fX$.
Then $\pd_\pm^fX$ are open and closed in $\pd X,$ so they are
manifolds with corners.
\label{mc4prop2}
\end{prop}

\begin{proof} As $\xi_-^f:\Xi_-^f\ra\pd X$ is a covering map by
Proposition \ref{mc4prop1}(b), $\xi_-^f(\Xi_-^f)$ is open in $\pd
X$. Since $\xi_-^f$ is proper and $\Xi_-^f,\pd X$ are Hausdorff,
$\xi_-^f(\Xi_-^f)$ is closed in $\pd X$. So $\pd_-^fX$, and hence
$\pd_+^fX=\pd X\sm\pd_-^fX$, are open and closed in~$\pd X$.
\end{proof}

We can characterize b-submersive morphisms $f:X\ra Y$ in Definition
\ref{mc3def2}(v) in terms of $\Xi^f_-,\xi^f_-$. The proof is an easy
exercise.

\begin{lem} Let\/ $f:X\ra Y$ be a smooth map of manifolds with
corners. Then $f$ is b-submersive if and only if\/ $\Xi^f_-=X\t_Y\pd
Y,$ so that\/ $\Xi^f_+=\es,$ and\/ $\xi^f_-:\Xi^f_-\ra\pd X$ is
injective.
\label{mc4lem}
\end{lem}

We now move on to our second way of describing how $f$ relates
$\pd^kX$ and $\pd^lY$. Equation \eq{mc2eq8} showed that if $X$ is a
manifold with corners then $X\mapsto\coprod_{i\ge 0}C_i(X)$ commutes
with products of manifolds. We will explain how to lift a smooth map
$f:X\ra Y$ up to a map $C(f):\coprod_{i\ge 0}C_i(X)\ra\coprod_{j\ge
0}C_j(Y)$ which is (in a generalized sense) smooth, and which is
functorial in a very strong sense.

\begin{dfn} Let $X,Y$ be smooth manifolds with corners and $f:X\ra
Y$ a smooth map. Define $C(f):\coprod_{i=0}^{\dim X} C_i(X)\ra
\coprod_{j=0}^{\dim Y}C_j(Y)$ by
\e
\begin{split}
&C(f):\bigl(x,\{\ti\be_1,\ldots,\ti\be_i\}\bigr)\longmapsto
\bigl(y,\{\be_1,\ldots,\be_j\}\bigr),\quad\text{where}\\
&\{\be_1,\ldots,\be_j\}\!=\!\bigl\{\be:\bigl(x,(y,\be)\bigr)\!\in\!
\Xi^f_-,\; \xi^f_-\bigl(x,(y,\be)\bigr)\!=\!(x,\ti\be_l),\;
l\!=\!1,\ldots,i\bigr\}.\!\!\!\!\!
\end{split}
\label{mc3eq2}
\e
\label{mc4def2}
\end{dfn}

\begin{dfn} Let $\{X_i:i\in I\}$ and $\{Y_j:j\in J\}$ be families of
manifolds, where $I,J$ are indexing sets. We do not assume that all
$X_i$ have the same dimension, or that all $Y_j$ have the same
dimension, so $\coprod_{i\in I}X_i$ and $\coprod_{j\in J}Y_j$ need
not be manifolds. We call a map $f:\coprod_{i\in I}X_i\ra
\coprod_{j\in J}Y_j$ {\it smooth\/} if $f$ is continuous, and for
all $i\in I$ and $j\in J$ the map
\begin{equation*}
f\vert_{X_i\cap f^{-1}(Y_j)}:X_i\cap f^{-1}(Y_j)\ra Y_j
\end{equation*}
is a smooth map of manifolds. Here $Y_j$ is an open and closed
subset of the topological space $\coprod_{j\in J}Y_j$, so $X_i\cap
f^{-1}(Y_j)$ is an open and closed subset of $X_i$ as $f$ is
continuous, and thus $X_i\cap f^{-1}(Y_j)$ is a manifold.
\label{mc4def3}
\end{dfn}

The next theorem, in part parallel to Theorem \ref{mc3thm}, gives
properties of these maps $C(f)$. The proofs are elementary. The
theorem basically says that mapping $X\mapsto \coprod_{i\ge
0}C_i(X)$ and $f\mapsto C(f)$ yields a {\it functor\/} which
preserves smoothness, composition, identities, boundaries $\pd X$,
immersions $i_X:\pd X\ra X$, and products and direct products of
smooth maps. Theorem \ref{mc6thm2} will also show that the functor
preserves strongly transverse fibre products.

\begin{thm} Let\/ $W,X,Y,Z$ be manifolds with corners.
\begin{itemize}
\setlength{\itemsep}{0pt}
\setlength{\parsep}{0pt}
\item[{\bf(i)}] If\/ $f:X\ra Y$ is smooth then\/
$C(f):\coprod_{i\ge 0}C_i(X)\ra\coprod_{j\ge 0}C_j(Y)$ is smooth
in the sense of Definition\/~{\rm\ref{mc4def3}}.
\item[{\bf(ii)}] If\/ $f:X\ra Y$ and\/ $g:Y\ra Z$ are smooth then
$C(g\ci f)=C(g)\ci C(f):\coprod_{i\ge 0}C_i(X)\ra\coprod_{k\ge
0}C_k(Z)$.
\item[{\bf(iii)}] $C(\id_X)=\id_{\coprod_{k\ge 0}C_k(X)}:
\coprod_{k\ge 0}C_k(X)\ra\coprod_{k\ge 0}C_k(X)$.
\item[{\bf(iv)}] The diffeomorphisms $C_k(\pd X)\cong\pd
C_k(X)$ in \eq{mc2eq9} identify
\begin{align*}
&\ts C(i_X):\coprod_{k\ge 0}C_k(\pd X)\longra \coprod_{k\ge 0}C_k(X)
\qquad\text{with}\\
&i_{\coprod_{k\ge 0}C_k(X)}:=\ts \coprod_{k\ge 0}i_{C_k(X)}=\ts
\coprod_{k\ge 0}\pd C_k(X)\longra \coprod_{k\ge 0}C_k(X).
\end{align*}
\item[{\bf(v)}] Let\/ $f:W\ra Y$ and\/ $g:X\ra Z$ be smooth
maps. Then \eq{mc2eq8} gives
\ea
\ts\coprod_{m\ge 0}C_m(W\t X)&\ts\cong\bigl[\coprod_{i\ge
0}C_i(W) \bigr]\t\bigl[\coprod_{j\ge 0}C_j(X)\bigr],
\nonumber\\
\ts\coprod_{n\ge 0}C_n(Y\t Z)&\ts\cong\bigl[\coprod_{k\ge
0}C_k(Y)\bigr]\t\bigl[\coprod_{l\ge 0}C_l(Z)\bigr].
\label{mc3eq3}
\ea
These identify $C(f\t g):\coprod_{m\ge 0}C_m(W\t
X)\ra\coprod_{n\ge 0}C_n(Y\t Z)$ with\/ $C(f)\t
C(g)\!:\!\bigl[\coprod_{i\ge 0}\!C_i(W)
\bigr]\!\t\!\bigl[\coprod_{j\ge 0}\!C_j(X) \bigr]\!\ra\!
\bigl[\coprod_{k\ge 0}\!C_k(Y)\bigr]\!\t\!\bigl[\coprod_{l\ge
0}\!C_l(Z)\bigr]$.
\item[{\bf(vi)}] Let\/ $f:X\ra Y$ and\/ $g:X\ra Z$ be smooth maps.
Then \eq{mc3eq3} identifies $C\bigl((f,g)\bigr):\coprod_{j\ge
0}C_j(X)\ra\coprod_{n\ge 0}C_n(Y\t Z)$ with\/
$\bigl(C(f),C(g)\bigr):\coprod_{j\ge
0}C_j(X)\ra\bigl[\coprod_{k\ge
0}C_k(Y)\bigr]\t\bigl[\coprod_{l\ge 0}C_l(Z)\bigr]$.
\item[{\bf(vii)}] Let\/ $f:X\ra Y$ be a b-submersive smooth map.
Then $C(f)$ maps $C_i(X)\ra \coprod_{j\ge 0}^iC_j(Y)$ for
all\/~$i\ge 0$.
\end{itemize}
\label{mc4thm}
\end{thm}

Curiously, there is a second way to define a map $\coprod_{i\ge
0}C_i(X)\ra\coprod_{j\ge 0}C_j(Y)$ with the same properties. Define
$\hat C(f):\coprod_{i\ge 0}C_i(X)\ra\coprod_{j\ge 0}C_j(Y)$ by
\begin{align*}
&\hat C(f):\bigl(x,\{\ti\be_1,\ldots,\ti\be_i\}\bigr)\longmapsto
\bigl(y,\{\be_1,\ldots,\be_j\}\bigr),\quad\text{where}\quad
\{\be_1,\ldots,\be_j\}=\\
&\bigl\{\be:\bigl(x,(y,\be)\bigr)\!\in\!
\Xi^f_-,\; \xi^f_-\bigl(x,(y,\be)\bigr)\!=\!(x,\ti\be_l),\;
1\!\le\! l\!\le\! i\bigr\}\!\cup\!
\{\bigl\{\be:\bigl(x,(y,\be)\bigr)\!\in\!
\Xi^f_+\bigr\}.
\end{align*}
Then the analogues of Theorems \ref{mc4thm} and \ref{mc6thm2} also
hold for $\hat C(f),\hat C(g),\ldots$.

\section{Submersions}
\label{mc5}

Definition \ref{mc3def2}(iv) defined {\it submersions\/} $f:X\ra Y$
between manifolds with corners $X,Y$. We show that submersions are
locally isomorphic to projections.

\begin{prop} Let\/ $X,Y$ be manifolds with corners, $f:X\ra Y$ a
submersion, and\/ $x\in X$ with\/ $f(x)=y$. Then there exist open
neighbourhoods\/ $X',Y'$ of\/ $x,y$ in\/ $X,Y$ with\/ $f(X')=Y',$ a
manifold\/ $Z'$ with corners with\/ $\dim X=\dim Y+\dim Z',$ and a
diffeomorphism $X'\cong Y'\t Z',$ such that\/ $f\vert_{X'}:X'\ra Y'$
is identified with\/~$\pi_{Y'}:Y'\t Z'\ra Y'$.
\label{mc5prop1}
\end{prop}

\begin{proof} Let $x\in X$ and $y=f(x)\in Y$, with $\dim X=m$,
$\dim Y=n$ and $x\in S^k(X)$, $y\in S^l(X)$. Choose charts
$(U,\phi)$, $(V,\psi)$ on $X,Y$ with $U,V$ open in $\R^m_k,\R^n_l$
and $0\in U$, $0\in V$ with $\phi(0)=x$, $\psi(0)=y$ and
$f\ci\phi(U)\subseteq\psi(V)$. Write $(x_1,\ldots,x_m)$,
$(y_1,\ldots,y_n)$ for the coordinates on $U,V$ respectively. Write
$\ti\be_i$ for the local boundary component $\phi_*(\{x_i=0\})$ for
$i=1,\ldots,k$, and $\be_j$ for the local boundary component
$\psi_*(\{y_j=0\})$ for~$j=1,\ldots,l$.

Lemma \ref{mc4lem} implies that $\bigl(x,(y,\be_j)\bigr)\in\Xi_-^f$
with $\xi_-^f\bigl(x,(y,\be_j)\bigr)=(x,\ti\be_{i_j})$ for each
$j=1,\ldots,l$ and some $i_j=1,\ldots,k$, and $i_1,\ldots,i_l$ are
distinct as $\xi_-^f$ is injective. Thus $l\le k$, and reordering
$x_1,\ldots,x_k$ if necessary we suppose that $i_j=j$. By
Proposition \ref{mc2prop6}(d), $\bigl(\psi(V),y_i\ci\psi^{-1}
\bigr)$ is a boundary defining function for $Y$ at $(y,\be_i)$ for
$i=1,\ldots,l$, so by Definition \ref{mc3def1}
$\bigl(f^{-1}(\psi(V)),y_i\ci\psi^{-1}\ci f\bigr)$ is a boundary
defining function for $X$ at $(x,\ti\be_i)$ for $i=1,\ldots,l$. But
$\bigl(\phi(U),x_i\ci\phi^{-1}\bigr)$ is also a boundary defining
function for $X$ at $(x,\ti\be_i)$, so by Proposition
\ref{mc2prop6}(b), making $U$ smaller if necessary we can suppose
that
\begin{equation*}
y_i\ci\psi^{-1}\ci f\ci\phi\equiv x_i\cdot g_i
\end{equation*}
on $U$, for some smooth $g_i:U\ra(0,\iy)$ and all $i=1,\ldots,l$.

Combining this with the surjectivity conditions in Definition
\ref{mc3def2}(iv), we see that we may choose alternative coordinates
$(\ti x_1,\ldots,\ti x_n)$ on an open neighbourhood $\ti U$ of 0 in
$U$ taking values in $\R^m_k$ and zero at 0, such that
\begin{equation*}
\ti x_i\equiv \begin{cases} y_i\ci\psi^{-1}\ci f\ci\phi,&
i=1,\ldots,l, \\ x_i, & i=l+1,\ldots,k, \\
y_{i-k+l}\ci\psi^{-1}\ci f\ci\phi,& i=k+1,\ldots,n-k+l, \\
\text{some function of $x_{k+1},\ldots,x_n$,} & i=n-k+l+1,\ldots,m.
\end{cases}
\end{equation*}
Choose small $\ep>0$ so that $[0,\ep)^k\t(-\ep,\ep)^{m-k} \subseteq
\ti U$ in coordinates $(\ti x_1,\ldots,\ti x_n)$. Then defining
$X'=\phi\bigl(\{(\ti x_1,\ldots,\ti x_m)\in
[0,\ep)^k\t(-\ep,\ep)^{m-k}\}\bigr)$, $Y'=\psi\bigl([0,\ep)^l\t
(-\ep,\ep)^{n-l}\bigr)$ and $Z'=[0,\ep)^{k-l}\t
(-\ep,\ep)^{m-k-n+l}$, the proposition follows.
\end{proof}

Submersions $f:X\ra Y$ are nicely compatible with the
boundaries~$\pd X,\pd Y$.

\begin{prop} Let\/ $f:X\ra Y$ be a submersion, and\/ $\pd
X=\pd_+^fX\amalg\pd_-^fX$ be as in Proposition {\rm\ref{mc4prop2}}.
Then $f_+=f\ci i_X\vert_{\pd_+^fX}:\pd_+^fX\ra Y$ is a submersion.
There is a natural submersion $\smash{f_-:\pd_-^fX\ra\pd Y}$
with\/~$\smash{f\ci i_X\vert_{\pd_-^fX}\equiv i_Y\ci f_-}$.
\label{mc5prop2}
\end{prop}

\begin{proof} Lemma \ref{mc4lem} shows that $\xi^f_-:\Xi^f_-=
X\t_Y\pd Y\ra\pd X$ is injective. Since $\pd_-^fX=\xi^f_-(\Xi^f_-)$,
it follows that $\xi^f_-:X\t_Y\pd Y\ra\pd_-^fX$ is invertible.
Define $f_-:\pd_-^fX\ra\pd Y$ by $f_-=\pi_{\pd Y}\ci(\xi^f_-)^{-1}$.
Then $i_Y\ci f_-=i_Y\ci\pi_{\pd Y}\ci(\xi^f_-)^{-1}
=f\ci\pi_X\ci(\xi^f_-)^{-1}=f\ci i_X\vert_{\smash{\pd_-^fX}}$, since
$i_Y\ci\pi_Y=f\ci\pi_X:X\t_Y\pd Y\ra Y$ and
$\pi_X\ci(\xi^f_-)^{-1}=i_X\vert_{\smash{\pd_-^fX}}: \pd_-^fX\ra X$
as $i_X\ci\xi^f_-\equiv\pi_X\vert_{\smash{\Xi^f_-}}$ by
Proposition~\ref{mc4prop1}(b).

It remains to check that the maps $f_\pm$ are submersions. Since
being a submersion is a local property, by Proposition
\ref{mc5prop1} it is enough to show $f_\pm$ are submersions when
$f:X\ra Y$ is a projection $\pi_{Y'}:Y'\t Z'\ra Y'$. We have a
natural isomorphism $\pd(Y'\t Z')\cong(\pd Y'\t Z')\amalg(Y'\t\pd
Z')$. It is easy to see that $\pd_+^f(Y'\t Z')\cong Y'\t\pd Z'$, so
that $f_+$ becomes the projection $Y'\t\pd Z'\ra Y'$ which is a
submersion, and that $\pd_-^f(Y'\t Z')\cong\pd Y'\t Z'$, so that
$f_-$ becomes the projection $\pd Y'\t Z'\ra\pd Y'$ which is a
submersion.
\end{proof}

Note that we can iterate this construction to decompose $\pd^kX$, so
that
\begin{equation*}
\pd^2X=\pd_+^{f_+}(\pd_+^fX)\amalg \pd_-^{f_+}(\pd_+^fX)\amalg
\pd_+^{f_-}(\pd_-^fX)\amalg \pd_-^{f_-}(\pd_-^fX),
\end{equation*}
for instance, and $f$ lifts to a submersion on every piece.

\section{Transversality and fibre products of manifolds}
\label{mc6}

Let $X,Y,Z$ be manifolds with corners and $f:X\ra Z$, $g:Y\ra Z$ be
smooth maps. From category theory, a {\it fibre product\/}
$X\t_{f,Z,g}Y$ in the category $\Manc$ consists of a manifold with
corners $W$ and smooth maps $\pi_X:W\ra X$, $\pi_Y:W\ra Y$ such that
$f\ci\pi_X=g\ci\pi_Y:W\ra Z$, satisfying the universal property that
if $W'$ is a manifold with corners and $\pi_X':W'\ra X$,
$\pi_Y':W'\ra Y$ are smooth maps with $f\ci\pi_X'=g\ci\pi_Y'$, then
there exists a unique smooth map $h:W'\ra W$ with $\pi_X'=\pi_X\ci
h$ and $\pi_Y'=\pi_Y\ci h$. We now give sufficient conditions for
fibre products of manifolds with corners to exist.

\begin{dfn} Let $X,Y,Z$ be manifolds with corners and $f:X\ra
Z$, $g:Y\ra Z$ be smooth maps. We call $f,g$ {\it transverse\/} if
the following holds. Suppose $x\in X$, $y\in Y$ and $z\in Z$ with
$f(x)=z=g(y)$, so that there are induced linear maps of tangent
spaces $\d f\vert_x:T_xX\ra T_zZ$ and $\d g\vert_y:T_yY\ra T_zZ$.
Let $x\in S^j(X)$, $y\in S^k(Y)$ and $z\in S^l(Z)$, so that as in
Definition \ref{mc3def2} $\d f\vert_x$ maps $T_x(S^j(X))\ra
T_z(S^l(Z))$ and $\d g\vert_y$ maps $T_y(S^k(Y))\ra T_z(S^l(Z))$.
Then we require that $T_zZ=\d f\vert_x(T_xX)+\d g\vert_y(T_yY)$ and
$T_z(S^l(Z))=\d f\vert_x(T_x(S^j(X)))+\d g\vert_y(T_y(S^k(Y)))$ for
all such $x,y,z$. From Definition \ref{mc3def2}, if one of $f,g$ is
a submersion then $f,g$ are automatically transverse.
\label{mc6def1}
\end{dfn}

\begin{rem} If $X,Y,Z$ are manifolds without boundary then $j=k=l=0$
in Definition \ref{mc6def1}, and both conditions reduce to the usual
definition $T_zZ=\d f\vert_x(T_xX)+\d g\vert_y(T_yY)$ of transverse
smooth maps. When $X,Y,Z$ are manifolds with corners we believe this
definition of transversality is new, since it depends heavily on our
definition of smooth maps which is also new.

Definition \ref{mc6def1} imposes two transversality conditions on
$f,g$ at $x,y,z$, the first on the corners $C_0(X)\cong
X,C_0(Y)\cong Y,C_0(Z)\cong Z$ of $X,Y,Z$ of largest dimension at
$x,y,z$, and the second on the corners $C_j(X)\cong S^j(X)$,
$C_k(Y)\cong S^k(X)$, $C_l(Z)\cong S^l(Z)$ (locally) of $X,Y,Z$ of
smallest dimension at~$x,y,z$.

One might think that to prove Theorem \ref{mc6thm1} one would need
to impose transversality conditions on corners
$C_a(X),C_b(Y),C_c(Y)$ of intermediate dimensions $0\le a\le j$,
$0\le b\le k$, $0\le c\le l$ as well. In fact these intermediate
conditions are implied by our definition of smooth maps, since the
requirement for $f,g$ to pull boundary defining functions back to
boundary defining functions is a kind of transversality condition at
the boundaries. One of the motivations for our definition of smooth
maps of manifolds with corners was to have a simple, not too
restrictive condition for the existence of fibre products.
\label{mc6rem1}
\end{rem}

\begin{rem} Margalef-Roig and Outerelo Dominguez \cite[\S
7.2]{MaOu} also define transversality of smooth maps between
manifolds with corners, and prove their own version of Theorem
\ref{mc6thm1} below. They work with Banach manifolds and $C^p$ maps
for $p=0,1,\ldots,\iy$. For finite-dimensional manifolds, their
notion of smooth map (`map of class $\iy$') corresponds to our
weakly smooth maps. However, their notion of {\it transversality\/}
\cite[Def.~7.2.1]{MaOu} is very restrictive.

In our notation, if $f:X\ra Z$ and $g:Y\ra Z$ are weakly smooth
maps, then $f,g$ are transverse in the sense of
\cite[Def.~7.2.1]{MaOu} if and only if whenever $x\in X$ and $y\in
Y$ with $f(x)=g(y)=z\in Z$ then $z\in Z^\ci$, and $x\in S^j(X)$,
$y\in S^k(Y)$ with $T_zZ=\d f\vert_x(T_x(S^j(X)))+\d
g\vert_y(T_y(S^k(Y)))$. In particular, $f(X)$ and $g(Y)$ cannot
intersect in the boundary strata $S^l(Z)$ for $l>0$ but only in the
interior $Z^\ci$, so in effect Margalef-Roig and Outerelo Dominguez
reduce to the case in which $\pd Z=\es$, and then their
\cite[Prop.~7.2.7]{MaOu} is a special case of Theorem \ref{mc6thm1}.
So, for example, $f,g$ are generally not transverse in the sense of
\cite[Def.~7.2.1]{MaOu} if one of $f,g$ is a submersion, or even
if~$f=\id_X:X\ra X=Z$.
\label{mc6rem2}
\end{rem}

For manifolds without boundary the following theorem is well-known,
as in Lang \cite[Prop.~II.4]{Lang}. For manifolds with corners
Margalef-Roig and Outerelo Dominguez \cite[Prop.~7.2.7]{MaOu} prove
it with a stricter notion of transversality, as above. We believe
this version is new. The proof is given in~\S\ref{mc8}.

\begin{thm} Suppose\/ $X,Y,Z$ are manifolds with corners and\/
$f:X\ra Z,$ $g:Y\ra Z$ are transverse smooth maps. Then there exists
a fibre product\/ $W=X\t_{f,Y,g}Z$ in the category $\Manc$ of
manifolds with corners, which is given by an explicit construction,
as follows.

As a topological space $W=\{(x,y)\in X\t Y:f(x)=g(y)\},$ with the
topology induced by the inclusion $W\subseteq X\t Y,$ and the
projections\/ $\pi_X:W\ra X$ and\/ $\pi_Y:W\ra Y$ map
$\pi_X:(x,y)\mapsto x,$ $\pi_Y:(x,y)\mapsto y$. Let\/ $n=\dim X+\dim
Y-\dim Z$, so that\/ $n\ge 0$ if\/ $W\ne\es$. The maximal atlas on\/
$W$ is the set of all charts\/ $(U,\phi),$ where\/
$U\subseteq\R^n_k$ is open and\/ $\phi:U\ra W$ is a homeomorphism
with a nonempty open set\/ $\phi(U)$ in $W,$ such that\/
$\pi_X\ci\phi:U\ra X$ and $\pi_Y\ci\phi:U\ra Y$ are smooth maps, and
for all\/ $u\in U$ with\/ $\phi(u)=(x,y),$ the following induced
linear map of real vector spaces is injective:
\e
\d(\pi_X\ci\phi)\vert_u\op\d(\pi_Y\ci\phi)\vert_u:T_uU=\R^n\longra
T_xX\op T_yY.
\label{mc6eq1}
\e
\label{mc6thm1}
\end{thm}

We note one important special case of Theorem \ref{mc6thm1}, the
intersection of submanifolds. Suppose $X,Y$ are embedded
submanifolds of $Z$, with inclusions $i:X\hookra Z$ and $j:Y\hookra
Z$. Then we say that $X,Y$ {\it intersect transversely\/} if the
smooth embeddings $i,j$ are transverse. Then the fibre product
$W=X\t_ZY$ is just the intersection $X\cap Y$ in $Z$, and Theorem
\ref{mc6thm1} shows that it is also an embedded submanifold of $Z$.
If $f,g$ are not transverse, then a fibre product $X\t_{f,Y,g}Z$ may
or may not exist in the category $\Manc$. Even if one exists, from
the point of view of derived differential geometry \cite{Spiv}, it
is in some sense the `wrong answer'. Here are some examples.

\begin{ex}{\bf(a)} The inclusion $i:\{0\}\ra\R$ is not transverse to
itself. A fibre product $\{0\}\t_{i,\R,i}\{0\}$ does exist in
$\Manc$ in this case, the point $\{0\}$. Note however that it does
not have the expected dimension: $\{0\}\t_{\R}\{0\}$ has dimension
0, but Theorem \ref{mc6thm1} predicts the dimension
$\dim\{0\}+\dim\{0\}-\dim\R=-1$.
\smallskip

\noindent{\bf(b)} Consider the smooth maps $f:\R\ra\R^2$ and
$g:\R\ra\R^2$ given by
\begin{equation*}
f(x)=(x,0)\quad\text{and}\quad g(x,y)=\begin{cases} (x,e^{-x^2}\sin
(\pi/x)), & x\ne 0, \\
(0,0), & x=0. \end{cases}
\end{equation*}
These are not transverse at $f(0)=g(0)=(0,0)$. The fibre product
does not exist in $\Manc$. To see this, note that the topological
fibre product $\R\t_{f,\R^2,g}\R$ is $\{(1/n,0):0\ne
n\in\Z\}\cup\{(0,0)\}$, which has no manifold structure.
\label{mc6ex1}
\end{ex}

In the general case of Theorem \ref{mc6thm1}, the description of
$\pd W$ in terms of $\pd X,\pd Y,\pd Z$ is rather complicated, as
can be seen from the proof in \S\ref{mc8}. Here are three cases in
which the expression simplifies. The proofs follow from the proof of
Theorem \ref{mc6thm1} in \S\ref{mc8}, or alternatively from equation
\eq{mc6eq9} below with $i=1$, since $\pd W\cong C_1(W)$ and $f,g$
are strongly transverse in each case.

\begin{prop} Let\/ $X,Y$ be manifolds with corners, and\/ $f:X\ra
Y$ a submersion. Then there is a canonical diffeomorphism
\e
\pd_-^fX\cong X\t_{f,Y,i_Y}\pd Y,
\label{mc6eq2}
\e
which identifies the submersions $\smash{f_-:\pd_-^fX\ra\pd Y}$
and\/~$\pi_{\pd Y}:X\t_{Y}\pd Y\ra\pd Y$.
\label{mc6prop1}
\end{prop}

\begin{prop} Let\/ $X,Y$ be manifolds with corners, $Z$ a manifold
without boundary, and\/ $f:X\ra Z,$ $g:Y\ra Z$ be transverse smooth
maps. Then $f\ci i_X:\pd X\ra Z,$ $g:Y\ra Z$ are transverse, and\/
$f:X\ra Z,$ $g\ci i_Y:\pd Y\ra Z$ are transverse, and there is a
canonical diffeomorphism
\e
\pd\bigl(X\t_{f,Z,g}Y\bigr)\cong \bigl(\pd X\t_{f\ci i_X,Z,g}Y\bigr)
\amalg \bigl(X\t_{f,Z,g\ci i_Y}\pd Y\bigr).
\label{mc6eq3}
\e
\label{mc6prop2}
\end{prop}

\begin{prop} Let\/ $X,Y,Z$ be manifolds with corners, $f:X\ra
Z$ a submersion and\/ $g:Y\ra Z$ smooth. Then there is a canonical
diffeomorphism
\e
\pd\bigl(X\t_{f,Z,g}Y\bigr)\cong \bigl(\pd_+^fX \t_{f_+,Z,g}Y\bigr)
\amalg \bigl(X\t_{f,Z,g\ci i_Y}\pd Y\bigr).
\label{mc6eq4}
\e
If both\/ $f,g$ are submersions there is also a canonical
diffeomorphism
\e
\begin{split}
\pd\bigl(&X\t_{f,Z,g}Y\bigr)\cong \\
&\bigl(\pd_+^fX \t_{f_+,Z,g}Y\bigr)\amalg \bigl(X \t_{f,Z,g_+}
\pd_+^gY\bigr)\amalg\bigl(\pd_-^fX\t_{f_-,\pd Z,g_-}\pd_-^gY\bigr).
\end{split}
\label{mc6eq5}
\e
Equation \eq{mc6eq4} also holds if\/ $f,g$ are transverse and $f$ is
b-submersive, and\/ \eq{mc6eq5} also holds if\/ $f,g$ are transverse
and both are b-submersive.
\label{mc6prop3}
\end{prop}

We will also discuss a stronger notion of transversality. To
introduce it we prove the following lemma.

\begin{lem} Let\/ $X,Y,Z$ be manifolds with corners, $f:X\ra Z$
and\/ $g:Y\ra Z$ be transverse smooth maps, and\/ $C(f),C(g)$ be as
in \eq{mc3eq2}. Suppose $(x,\{\be_1,\ldots,\be_j\})\in C_j(X)$ and\/
$(y,\{\ti\be_1,\ldots,\ti\be_k\})\in C_k(Y)$ with\/ $C(f)
(x,\{\be_1,\ab\ldots,\ab\be_j\})\ab=C(g)(y,\{\ti\be_1,\ldots,
\ti\be_k\})= (z,\{\dot\be_1,\ldots,\dot\be_l\})$ in\/ $C_l(Z)$.
Then\/~$j+k\ge l$.
\label{mc6lem}
\end{lem}

\begin{proof} Since $C(f)(x,\{\be_1,\ldots,\be_j\})=
(z,\{\dot\be_1,\ldots,\dot\be_l\})$ it follows that $\d f\vert_x$
maps the vector subspace $T_x\be_1\cap\cdots\cap T_x\be_j$ in $T_xX$
to the vector subspace $T_z\dot\be_1\cap\cdots\cap T_z\dot\be_l$ in
$T_zZ$, as the restriction of $\d f\vert_x$ to $T_x\be_1\cap
\cdots\cap T_x\be_j$ is naturally identified with $\d
C(f)\vert_{(x,\{\be_1,\ldots,\be_j\})}$. Similarly, $\d g\vert_y$
maps $T_y\ti\be_1\cap\cdots\cap T_y\ti\be_k$ in $T_yY$ to
$T_z\dot\be_1\cap\cdots\cap T_z\dot\be_l$ in $T_zZ$. Since $f,g$ are
transverse, we have $T_zZ=\d f\vert_x(T_xX)+\d g\vert_y(T_yY)$.
Passing to the quotients $T_xX/(T_x\be_1\cap\cdots\cap
T_x\be_j),\ab\ldots,T_zZ/(T_z\dot\be_1\cap\cdots\cap T_z\dot\be_l)$
and using the facts above shows that
\e
\begin{split}
(\d f\vert_x)_*\bigl(T_xX/(T_x\be_1\cap\cdots\cap T_x\be_j)\bigr)+
(\d g\vert_y)_*\bigl(T_yY/(T_y\ti\be_1\cap\cdots\cap
T_y\ti\be_k)\bigr)&\\
= T_zZ/(T_z\dot\be_1\cap\cdots\cap T_z\dot\be_l)&.
\end{split}
\label{mc6eq6}
\e
As the vector spaces in \eq{mc6eq6} have dimensions $j,k,l$, it
follows that~$j+k\ge l$.
\end{proof}

\begin{dfn} Let $X,Y,Z$ be manifolds with corners and $f:X\ra
Z$, $g:Y\ra Z$ be smooth maps. We call $f,g$ {\it strongly
transverse\/} if they are transverse, and whenever there are points
in $C_j(X),C_k(Y),C_l(Z)$ with
\e
C(f)(x,\{\be_1,\ldots,\be_j\})=C(g)(y,\{\ti\be_1,\ldots,
\ti\be_k\})=(z,\{\dot\be_1,\ldots,\dot\be_l\})
\label{mc6eq7}
\e
we have either $j+k>l$ or $j=k=l=0$. That is, in Lemma \ref{mc6lem},
equality in $j+k\ge l$ is allowed only if~$j=k=l=0$.

Suppose $f,g$ are smooth with $f$ a submersion. Then $f,g$ are
automatically transverse, as in Definition \ref{mc6def1}, and in
\eq{mc6eq7}, Theorem \ref{mc4thm}(v) implies that $j\ge l$. Hence if
$k>0$ then $j+k>l$. If $k=0$ then $l=0$ as $C(g)$ maps $C_0(Y)\ra
C_0(Z)$, so either $j+k>l$ or $j=k=l=0$. So $f,g$ are strongly
transverse. Also $f,g$ are strongly transverse if $f,g$ are smooth
with $g$ a submersion, or if $f,g$ are transverse and~$\pd Z=\es$.
\label{mc6def2}
\end{dfn}

In the situation of Theorem \ref{mc6thm1} we have a Cartesian square
\e
\begin{gathered}
\xymatrix@R=10pt{ W \ar[r]_{\pi_Y} \ar[d]^{\pi_X} & Y \ar[d]_g
\\ X \ar[r]^f & Z,}
\end{gathered}
\quad \quad
\begin{aligned}
&\text{which induces}\\
&\text{a commutative}\\
&\text{square}
\end{aligned}
\quad
\begin{gathered}
\xymatrix@R=10pt{ \coprod_{i\ge 0}C_i(W) \ar[r]_{C(\pi_Y)}
\ar[d]^{C(\pi_X)} & \coprod_{k\ge 0}C_k(Y) \ar[d]_{C(g)} \\
\coprod_{j\ge 0}C_j(X) \ar[r]^{C(f)} & \coprod_{l\ge 0}C_l(Z).}
\end{gathered}
\label{mc6eq8}
\e
Since as in Theorem \ref{mc4thm} the transformation $X\mapsto
\coprod_{i\ge 0}C_i(X)$, $f\mapsto C(f)$ has very good functorial
properties, it is natural to wonder whether the right hand square in
\eq{mc6eq8} is also Cartesian. The answer is yes if and only if
$f,g$ are strongly transverse. The following theorem will be proved
in~\S\ref{mc9}.

\begin{thm} Let\/ $X,Y,Z$ be manifolds with corners, and\/ $f:X\ra
Z,$ $g:Y\ra Z$ be strongly transverse smooth maps, and write\/ $W$
for the fibre product\/ $X\t_{f,Z,g}Y$ in Theorem\/
{\rm\ref{mc6thm1}}. Then there is a canonical diffeomorphism
\e
\begin{gathered}
C_i(W)\cong \coprod_{\begin{subarray}{l}j,k,l\ge 0:\\
i=j+k-l\end{subarray}}
\begin{aligned}[t]
\bigl(C_j(X)\cap C(f)^{-1}(C_l(Z))\bigr) \t_{C(f),C_l(Z),C(g)}&\\
\bigl(C_k(Y)\cap C(g)^{-1}(C_l(Z))\bigr)&
\end{aligned}
\end{gathered}
\label{mc6eq9}
\e
for all\/ $i\ge 0,$ where the fibre products are all transverse and
so exist. Hence
\e
\ts \coprod_{i\ge 0}C_i(W)\cong \coprod_{j\ge
0}C_j(X)\t_{C(f),\coprod_{l\ge 0}C_l(Z),C(g)} \coprod_{k\ge
0}C_k(Y).
\label{mc6eq10}
\e
Here the right hand commutative square in\/ \eq{mc6eq8} induces a
map from the left hand side of\/ \eq{mc6eq10} to the right hand
side, which gives the identification\/~\eq{mc6eq10}.
\label{mc6thm2}
\end{thm}

Suppose $f:X\ra Z$ and $g:Y\ra Z$ are transverse, but not strongly
transverse. Then by Definition \ref{mc6def2} there exist points in
$C_j(X),C_k(Y),C_l(Z)$ satisfying \eq{mc6eq7} with $j+k=l$ but
$j,k,l$ not all zero. These give a point in the right hand side of
\eq{mc6eq9} with $i=0$ which does not lie in the image of $C_0(W)$
under the natural map, since $C_0(W)$ maps to $C_0(X),C_0(Y)$ and so
cannot map to $C_j(X),C_k(Y)$ as $j,k$ are not both zero. Thus
\eq{mc6eq9} and \eq{mc6eq10} are false if $f,g$ are transverse but
not strongly transverse.

Here is an example of $f,g$ which are transverse but not strongly
transverse.

\begin{ex} Define smooth maps $f:[0,\iy)\ra[0,\iy)^2$ by
$f(x)=(x,2x)$ and $g:[0,\iy)\ra[0,\iy)^2$ by $g(y)=(2y,y)$. Then
$f(0)=g(0)=(0,0)$. We have
\begin{gather*}
\d f\vert_0\bigl(T_0[0,\iy)\bigr)+\d g\vert_0\bigl(T_0[0,\iy)\bigr)=
\an{(1,2)}_\R+\an{(2,1)}_\R=\R^2=T_{(0,0)}[0,\iy)^2,\\
\d f\vert_0\bigl(T_0\bigl(S^0([0,\iy))\bigr)\bigr) +\d
g\vert_0\bigl(T_0\bigl(S^0([0,\iy))\bigr)\bigr)=\{0\}=
T_{(0,0)}\bigl(S^0([0,\iy)^2)\bigr),
\end{gather*}
so $f,g$ are transverse. However we have
\begin{equation*}
C(f)\bigl(0,\bigl\{\{x=0\}\bigr\}\bigr)=C(g)\bigl(0,\bigl\{\{y=0\}
\bigr\})=\bigl((0,0),\bigl\{\{x=0\},\{y=0\}\bigr\}\bigr),
\end{equation*}
with $j=k=1$ and $l=2$, so $f,g$ are not strongly transverse. The
fibre product $W=[0,\iy)_{f,[0,\iy)^2,g}[0,\iy)$ is a single point
$\{0\}$. In \eq{mc6eq9} when $i=0$ the l.h.s.\ is one point, and the
r.h.s.\ is two points, one from $j\!=\!k\!=\!l\!=\!0$ and one from
$j\!=\!k\!=\!1$, $l\!=\!2$, so \eq{mc6eq9} does not hold. For
$i\!\ne\! 0$, both sides of \eq{mc6eq9} are empty.
\label{mc6ex2}
\end{ex}

The distinction between transversality and strong transversality
will be important in \cite{Joyc}. There we will define a 2-category
$\dManc$ of {\it d-manifolds with corners}, a `derived'
generalization of manifolds with corners, which contains the
1-category $\Manc$ of manifolds with corners as a full discrete
2-subcategory. If $X,Y,Z$ are manifolds with corners and $f:X\ra Z$,
$g:Y\ra Z$ are smooth then a 2-category fibre product
$(X\t_ZY)_\dManc$ exists in $\dManc$. If $f,g$ are transverse then a
1-category fibre product $(X\t_ZY)_\Manc$ also exists in
$\Manc\subset\dManc$ by Theorem \ref{mc6thm1}. However,
$(X\t_ZY)_\dManc$ and $(X\t_ZY)_\Manc$ coincide if and only if $f,g$
are strongly transverse.

\section{Orientations and orientation conventions}
\label{mc7}

Orientations are discussed in \cite[\S I.1]{KoNo} and~\cite[\S
VIII.3]{Lang}.

\begin{dfn} Let $X$ be an $n$-manifold and $E\ra X$ a vector bundle of
rank $k$. The {\it frame bundle\/} $F(E)$ is
\begin{equation*}
F(E)=\bigl\{(x,e_1,\ldots,e_k):\text{$x\in X$, $(e_1,\ldots,e_k)$ is a
basis for $E\vert_x\cong\R^k$}\bigr\}.
\end{equation*}
It is a manifold of dimension $n+k^2$. Define an action of
$\GL(k,\R)$ on $F(E)$ by $(A_{ij})_{i,j=1}^k:(x,e_1,\ldots,e_k)
\mapsto\bigl(x,\sum_{j=1}^kA_{1j}e_j,\ldots,\sum_jA_{kj}e_j\bigr)$.
This action is smooth and free, and makes $F(E)$ into a principal
$\GL(k,\R)$-bundle over $X$, with projection $\pi:F(E)\ra X$ given
by~$\pi:(x,e_1,\ldots,e_k)\mapsto x$.

Write $\GL_+(k,\R)$ for the subgroup of $A\in\GL(k,\R)$ with $\det
A>0$. It is a normal subgroup of $\GL(k,\R)$ of index 2, and we
identify the quotient subgroup $\GL(k,\R)/\GL_+(k,\R)$ with $\{\pm
1\}$ by $A\GL_+(k,\R)\mapsto\det A/\md{\det A}$. The {\it
orientation bundle\/} $\Or(E)$ of $E$ is $F(E)/\GL_+(k,\R)$. It is a
principal $\GL(k,\R)/\GL_+(k,\R)=\{\pm 1\}$-bundle over $X$. Points
of the fibre of $\Or(E)$ over $x\in X$ are equivalence classes of
bases $(e_1,\ldots,e_k)$ for $E\vert_x$, where two bases are
equivalent if they are related by a $k\t k$ matrix with positive
determinant.

An {\it orientation\/ $o_E$ for the fibres of\/} $E$ is a continuous
section $o_E:X\ra\Or(E)$ of $\Or(E)$. The pair $(E,o_E)$ is called
an {\it oriented vector bundle\/} on $X$. If $E\ra X$, $F\ra X$ are
vector bundles on $X$ of ranks $k,l$ and $o_E,o_F$ are orientations
on the fibres of $E,F$, we define the {\it direct sum orientation\/}
$o_{E\op F}=o_E\op o_F$ on the fibres of $E\op F$ by saying that if
$x\in X$, $(e_1,\ldots,e_k)$ is an oriented basis for $E\vert_x$ and
$(f_1,\ldots,f_l)$ is an oriented basis for $F\vert_x$, then
$(e_1,\ldots,e_k,f_1,\ldots,f_l)$ is an oriented basis for~$(E\op
F)\vert_x$.

An {\it orientation\/ $o_X$ for\/} $X$ is an orientation for the
fibres of the tangent bundle $TX\ra X$. An {\it oriented manifold\/}
$(X,o_X)$ is a manifold $X$ with an orientation $o_X$. Usually we
leave $o_X$ implicit, and call $X$ an oriented manifold. If $o_X$ is
an orientation on $X$ then the {\it opposite orientation\/} on $X$
is $-o_X$, where $o_X:X\ra\Or(TX)$ is a section,
$-1:\Or(TX)\ra\Or(TX)$ comes from the principal $\{\pm 1\}$-action
on $\Or(TX)$, and $-o_X=-1\ci o_X$ is the composition. When $X$ is
an oriented manifold, we write $-X$ for $X$ with the opposite
orientation.
\label{mc7def1}
\end{dfn}

We shall consider issues to do with orientations on manifolds with
corners, and orientations on fibre products of manifolds. To do
this, we need {\it orientation conventions\/} to say how to orient
boundaries $\pd X$ and fibre products $X\t_ZY$ of oriented manifolds
$X,Y,Z$. Our conventions generalize those of Fukaya, Oh, Ohta and
Ono \cite[Conv.~45.1]{FOOO}, who restrict to $f,g$ submersions.

\begin{conv}{\bf(a)} Let $(X,o_X)$ be an oriented manifold with
corners. Define $o_{\pd X}$ to be the unique orientation on $\pd X$
such that
\e
i_X^*(TX)\cong\R_{\rm out}\op T(\pd X)
\label{mc7eq1}
\e
is an isomorphism of oriented vector bundles over $\pd X$, where
$i_X^*(TX),T(\pd X)$ are oriented by $o_X,o_{\pd X}$, and $\R_{\rm
out}$ is oriented by an outward-pointing normal vector to $\pd X$ in
$X$, and the r.h.s.\ of \eq{mc7eq1} has the direct sum orientation.
\smallskip

\noindent{\bf(b)} Let $(X,o_X),(Y,o_Y),(Z,o_Z)$ be oriented
manifolds with corners, and $f:X\ra Z$, $g:Y\ra Z$ be transverse
smooth maps, so that a fibre product $W=X\t_ZY$ exists in $\Manc$ by
Theorem \ref{mc6thm1}. Then we have an exact sequence of vector
bundles over $W$
\e
\xymatrix@C=8.5pt{0 \ar[r] & TW \ar[rrr]^(0.35){\d\pi_X\op\d\pi_Y}
&&& \pi_X^*(TX)\op \pi_Y^*(TY) \ar[rrrr]^(0.55){\pi_X^*(\d
f)-\pi_Y^*(\d g)} &&&& (f\ci\pi_X)^*(TZ) \ar[r] & 0.}
\label{mc7eq2}
\e
Choosing a splitting of \eq{mc7eq2} induces an isomorphism of vector
bundles
\e
TW \op (f\ci\pi_X)^*(TZ)\cong \pi_X^*(TX)\op \pi_Y^*(TY).
\label{mc7eq3}
\e
Define $o_W$ to be the unique orientation on $W$ such that the
direct sum orientations in \eq{mc7eq3} induced by $o_W,o_Z,o_X,o_Y$
differ by a factor~$(-1)^{\dim Y\dim Z}$.

Here are two was to rewrite this convention in special cases.
Firstly, suppose $f$ is a submersion. Then $\d f:TX\ra f^*(TZ)$ is
surjective, so by splitting the exact sequence $0\ra\Ker\d f\ra
TX\,{\buildrel\d f\over\longra}\,f^*(TZ)\ra 0$ we obtain an
isomorphism
\e
TX\cong \Ker\d f\op f^*(TZ).
\label{mc7eq4}
\e
Give the vector bundle $\Ker\d f\ra X$ the unique orientation such
that \eq{mc7eq4} is an isomorphism of oriented vector bundles, where
$TX,f^*(TZ)$ are oriented using $o_X,o_Z$. As $f:X\ra Z$ is a
submersion so is $\pi_Y:W\ra Y$, and $\d\pi_X$ induces an
isomorphism $\Ker(\d\pi_Y)\ra\pi_X^*(\Ker\d f)$. Thus we have an
exact sequence
\begin{equation*}
\xymatrix@C=15pt{0 \ar[r] & \pi_X^*(\Ker\d f)
\ar[rr]^(0.6){(\d\pi_X)^{-1}} && TW \ar[rr]^(0.45){\d\pi_Y} && \pi_Y^*(TY)
\ar[r] & 0.}
\end{equation*}
Splitting this gives an isomorphism
\e
TW\cong \pi_X^*(\Ker\d f) \op \pi_Y^*(TY).
\label{mc7eq5}
\e
The orientation on $W$ makes \eq{mc7eq5} into an isomorphism of
oriented vector bundles, using $o_Y$ and the orientation on $\Ker\d
f$ to orient the right hand side.

Secondly, let $g$ be a submersion. Then as for
\eq{mc7eq4}--\eq{mc7eq5} we have isomorphisms
\e
TY\cong g^*(TZ)\op\Ker\d g\quad\text{and}\quad TW\cong
\pi_X^*(TX)\op\pi_Y^*(\Ker\d g).
\label{mc7eq6}
\e
We use the first equation of \eq{mc7eq6} to define an orientation on
the fibres of $\Ker\d g$, and the second to define an orientation
on~$W$.
\label{mc7conv}
\end{conv}

If $X$ is an oriented manifold with corners then by induction
Convention \ref{mc7conv}(a) gives orientations on $\pd^kX$ for all
$k=0,1,2,\ldots$. Now Definition \ref{mc2def6} defined a smooth,
free action of $S_k$ on $\pd^kX$ for each $k$. By considering local
models $\R^n_l$ it is easy to see that the action of each $\si\in
S_k$ multiplies the orientation on $\pd^kX$ by $\sign(\si)=\pm 1$.
Since $C_k(X)\cong\pd^kX/S_k$ by \eq{mc2eq4} and $S_k$ does not
preserve orientations for $k\ge 2$, we see that $C_k(X)$ {\it does
not have a natural orientation for\/} $k\ge 2$. We show by example
that $C_k(X)$ need not even be orientable.

\begin{ex} Let $X$ be the 4-manifold with corners $\bigl({\cal
S}^2\t[0,\iy)^2\bigr)/\Z_2$, where
\begin{equation*}
{\cal S}^2\t[0,\iy)^2=\bigl\{(x_2,x_2,x_3,y_1,y_2):\text{$x_j,y_j\in\R$,
$x_1^2+x_2^2+x_3^2=1$, $y_1,y_2\ge 0$}\bigr\},
\end{equation*}
and $\Z_2=\an{\si}$ acts freely on $X$ by
\begin{equation*}
\si:(x_2,x_2,x_3,y_1,y_2)\mapsto(-x_1,-x_2,-x_3,y_2,y_1).
\end{equation*}
There is an orientation on ${\cal S}^2\t[0,\iy)^2$ which is
invariant under $\Z_2$, and so descends to $X$. We have
diffeomorphisms
\begin{equation*}
\pd X\cong C_1(X)\cong {\cal S}^2\t[0,\iy),\quad \pd^2X\cong {\cal S}^2,
\quad C_2(X)\cong \RP^2,
\end{equation*}
and $\pd^kX\!=\!C_k(X)\!=\!\es$ for $k\!>\!2$. Thus $X$ is oriented,
but $C_2(X)$ is not orientable.
\label{mc7ex}
\end{ex}

Given any canonical diffeomorphism between expressions involving
boundaries and fibre products of oriented manifolds with corners, we
can use Convention \ref{mc7conv} to define orientations on each
side. These will be related by some sign $\pm 1$, which we can try
to compute. Here is how to add signs to~\eq{mc6eq2}--\eq{mc6eq5}.

\begin{prop} In Propositions {\rm\ref{mc6prop1}, \ref{mc6prop2}}
and\/ {\rm\ref{mc6prop3},} suppose $X,Y,Z$ are oriented. Then in
oriented manifolds, equations \eq{mc6eq2}--\eq{mc6eq5} respectively
become
\ea
\pd_-^fX&\cong (-1)^{\dim X+\dim Y}X\t_{f,Y,i_Y}\pd Y,
\label{mc7eq7}\\
\pd\bigl(X\t_{f,Z,g}Y\bigr)&\cong \bigl(\pd X\t_{f\ci
i_X,Z,g}Y\bigr) \amalg (-1)^{\dim X+\dim Z}\bigl(X\t_{f,Z,g\ci
i_Y}\pd Y\bigr),
\label{mc7eq8}\\
\pd\bigl(X\t_{f,Z,g}Y\bigr)&\cong \bigl(\pd_+^fX \t_{f_+,Z,g}Y\bigr)
\amalg (-1)^{\dim X+\dim Z}\bigl(X\t_{f,Z,g\ci i_Y}\pd Y\bigr),
\label{mc7eq9}\\
\begin{split}
\pd\bigl(X\t_{f,Z,g}Y\bigr)&\cong\bigl(\pd_+^fX \t_{f_+,Z,g}Y\bigr)
\amalg (-1)^{\dim X+\dim Z}\bigl(X\t_{f,Z,g_+}\pd_+^gY\bigr)\\
&\qquad \amalg\bigl(\pd_-^fX\t_{f_-,\pd Z,g_-}\pd_-^gY\bigr).
\end{split}
\label{mc7eq10}
\ea
\label{mc7prop1}
\end{prop}

Here are some more identities involving only fibre products:

\begin{prop}{\bf(a)} If\/ $f:X\ra Z,$ $g:Y\ra Z$ are transverse
smooth maps of oriented manifolds with corners then in oriented
manifolds we have
\e
X\t_{f,Z,g}Y\cong(-1)^{(\dim X-\dim Z)(\dim Y-\dim Z)}Y\t_{g,Z,f}X.
\label{mc7eq11}
\e

\noindent{\bf(b)} If\/ $d:V\ra Y,$ $e:W\ra Y,$ $f:W\ra Z,$ $g:X\ra
Z$ are smooth maps of oriented manifolds with corners then in
oriented manifolds we have
\e
V\t_{d,Y,e\ci\pi_W}\bigl(W\t_{f,Z,g}X\bigr)\cong
\bigl(V\t_{d,Y,e}W\bigr)\t_{f\ci\pi_W,Z,g}X,
\label{mc7eq12}
\e
provided all four fibre products are transverse.

\noindent{\bf(c)} If\/ $d:V\ra Y,$ $e:V\ra Z,$ $f:W\ra Y,$ $g:X\ra
Z$ are smooth maps of oriented manifolds with corners then in
oriented manifolds we have
\e
\begin{split}
&V\t_{(d,e),Y\t Z,f\t g}(W\t X)\cong \\
&\quad(-1)^{\dim Z(\dim Y+\dim W)}
(V\t_{d,Y,f}W)\t_{e\ci\pi_V,Z,g}X,
\end{split}
\label{mc7eq13}
\e
provided all three fibre products are transverse.
\label{mc7prop2}
\end{prop}

\begin{rem}{\bf(i)} Equations \eq{mc7eq8}, \eq{mc7eq12} and
\eq{mc7eq13} can be found in Fukaya et al.\ \cite[Lem.~45.3]{FOOO}
for the case of Kuranishi spaces.
\smallskip

\noindent{\bf(ii)} The proofs of Propositions \ref{mc7prop1} and
\ref{mc7prop2} are elementary calculations starting from Convention
\ref{mc7conv}. Here is a way to make these calculations easier. For
simplicity, assume all the smooth maps involved are submersions. By
Proposition \ref{mc5prop1}, submersions are locally projections.
Since identities like \eq{mc7eq7}--\eq{mc7eq13} are local, it is
enough to prove the identities for projections.

Let $M,N,Z$ be oriented manifolds with corners, of dimensions
$m,n,z$. Set $X=M\t Z$ and $Y=Z\t N$, with the product orientations,
and define $f:X\ra Z$, $g:Y\ra Z$ by $f=\pi_Z=g$. Convention
\ref{mc7conv}(b) is arranged so that $W\cong M\t Z\t N$ holds in
oriented manifolds. Exchanging the order in a product of oriented
manifolds yields $X\t Y\cong(-1)^{\dim X\dim Y}Y\t X$. Thus to
compute the sign in \eq{mc7eq11}, for instance, note that
\begin{align*}
&X\t_{f,Z,g}Y=(M\t Z)\t_Z(Z\t N)\cong M\t Z\t N,\\
&Y\t_{g,Z,f}X=(Z\t N)\t_Z(M\t Z)\cong(-1)^{(m+n)z}
(N\t Z)\t_Z(Z\t M)\\
&\quad\cong(-1)^{(m+n)z}N\t Z\t M\cong(-1)^{(m+n)z}(-1)^{mn+mz+nz}M\t Z\t N,
\end{align*}
and then substitute in $m=\dim X-\dim Z$, $n=\dim Y-\dim Z$.
\smallskip

\noindent{\bf(iii)} Ramshaw and Basch \cite{RaBa} prove that there
is a {\it unique\/} orientation convention for transverse fibre
products of manifolds without boundary satisfying the three
conditions: (A) if $X,Y$ are oriented then $X\t_{\{0\}}Y\cong X\t Y$
in oriented manifolds, where $X\t Y$ has the product orientation
from $T(X\t Y)\cong \pi_X^*(TX)\op \pi_Y^*(TY)$; (B) if $f:X\ra Y$
is a smooth map of oriented manifolds then $X\cong Y\t_{\id_Y,Y,f}X$
in oriented manifolds; and (C) equation \eq{mc7eq12} holds.
Convention \ref{mc7conv}(b) satisfies (A)--(C), and so agrees with
that of~\cite{RaBa}.
\label{mc7rem}
\end{rem}

\section{Proof of Theorem \ref{mc6thm1}}
\label{mc8}

Theorem \ref{mc6thm1} follows from the next two propositions. In the
proof we assume Theorem \ref{mc6thm1} for manifolds without
boundary, since this is well known, as in Lang
\cite[Prop.~II.4]{Lang} for instance. The difference between {\it
transverse\/} and {\it strongly transverse\/} $f,g$ in \S\ref{mc6}
appears in part (C) in the proof below: transverse $f,g$ are
strongly transverse if and only if there are no $\sim$-equivalence
classes $E$ of type~(b).

\begin{prop} Let\/ $X,Y,Z$ be manifolds with corners, and\/
$f:X\ra Z,$ $g:Y\ra Z$ be transverse smooth maps. Then the
construction of Theorem\/ {\rm\ref{mc6thm1}} defines a manifold with
corners $W,$ with\/ $\dim W=\dim X+\dim Y-\dim Z$ if\/ $W\ne\es,$
and the maps $\pi_X:W\ra X,$ $\pi_Y:W\ra Y$ are smooth.
\label{mc8prop1}
\end{prop}

\begin{proof} If $W=\es$ the proposition is trivial, so suppose
$W\ne\es$. Then $n=\dim X+\dim Y-\dim Z\ge 0$, since $f,g$ are
transverse. Let $(x,y)\in W$, so that $x\in X$ and $y\in Y$ with
$f(x)=g(y)=z$ in $Z$. We will first construct a chart $(U,\phi)$ on
$W$ satisfying the conditions of Theorem \ref{mc6thm1}, with $U$
open in $\R^n_d$ and $0\in U$ with~$\phi(0)=(x,y)$.

Choose charts $(R,\th),(S,\psi),(T,\xi)$ on $X,Y,Z$ respectively
with $0\in R,S,T$ and $\th(0)=x$, $\psi(0)=y$, $\xi(0)=z$, where
$R,S,T$ are open in $\R^k_a,\R^l_b,\R^m_c$ with $k=\dim X$, $l=\dim
Y$, $m=\dim Z$. Making $R,S$ smaller if necessary suppose
$f\ci\th(R),g\ci\psi(S)\subseteq\xi(T)$. Then $\ti f=\xi^{-1}\ci
f\ci\th:R\ra T$ and $\ti g=\xi^{-1}\ci g\ci\psi:S\ra T$ are smooth
maps between subsets of $\R^k,\R^l,\R^m$ in the sense of Definition
\ref{mc2def1}. So by definition we can choose open subsets $\hat
R\subseteq\R^k$, $\hat S\subseteq\R^l$, $\hat T\subseteq\R^m$ with
$R=\hat R\cap\R^k_a$, $S=\hat S\cap\R^l_b$, $T=\hat T\cap \R^m_c$
and smooth maps $\hat f:\hat R\ra\hat T$, $\hat g:\hat S\ra\hat T$
with $\hat f\vert_R=\ti f$, $\hat g\vert_S=\ti g$.

Now $f,g$ are transverse, so $\ti f,\ti g$ are transverse on $R,S$,
and as this is an open condition, by making $\hat R,\hat S$ smaller
if necessary we can make $\hat f,\hat g$ transverse. Since $\hat
f:\hat R\ra\hat T$, $\hat g:\hat S\ra\hat T$ are transverse smooth
maps of manifolds without boundary, by \cite[Prop.~II.4]{Lang} the
fibre product $\hat V=\hat R\t_{\hat f,\hat T,\hat g}\hat S$ exists
as an $n$-manifold without boundary, and it is also easy to show
that charts on $\hat V$ are characterized by the injectivity of
\eq{mc6eq1}. Define $V=\bigl\{(r,s)\in\hat V:r\in R$, $s\in
S\bigr\}$. We will show that near $(0,0)\in V$, the embedding of $V$
in $\hat V$ is modelled on the inclusion of $\R^n_d$ in $\R^n$, so
that $V$ is a manifold with corners.

The local boundary components of $R\subseteq\R^k_a$ at 0 are
$\{r_i=0\}$ for $i=1,\ldots,a$, where $(r_1,\ldots,r_k)$ are the
coordinates on $R$ and $\hat R$. Write $\be_i=\th_*(\{r_i=0\})$ for
the corresponding local boundary component of $X$ at $x$. Then
$\be_1,\ldots,\be_a$ are the local boundary components of $X$ at
$x$. Similarly, write $(s_1,\ldots,s_l)$ for coordinates on $S,\hat
S$ and $\ti\be_1,\ldots,\ti\be_b$ for the local boundary components
of $Y$ at $y$, where $\ti\be_i=\psi_*(\{s_i=0\})$, and
$(t_1,\ldots,t_m)$ for coordinates on $T,\hat T$ and
$\dot\be_1,\ldots,\dot\be_c$ for the local boundary components of
$Z$ at $z$, where~$\dot\be_i=\xi_*(\{t_i=0\})$.

Define subsets $P^f,P^g\subseteq\{1,\ldots,c\}$ by
$P^f=\bigl\{i:(x,(z,\dot\be_i))\in\Xi_-^f\bigr\}$ and
$P^g=\bigl\{i:(y,(z,\dot\be_i))\in\Xi_-^g\bigr\}$. Define maps
$\Pi^f:P^f\ra\{1,\ldots,a\}$ and $\Pi^g:P^g\ra\{1,\ldots,b\}$ by
$\Pi^f(i)=j$ if $\xi_-^f\bigl(x,(z,\dot\be_i)\bigr)=(x,\be_j)$ and
$\Pi^g(i)=j$ if $\xi_-^g\bigl(y,(z,\dot\be_i)\bigr)=(y,\ti\be_j)$.
We can express the maps $C(f),C(g)$ of \eq{mc3eq2} over $x,y,z$ as
follows: if $A\subseteq\{1,\ldots,a\}$ and
$B\subseteq\{1,\ldots,b\}$ then
\e
\begin{split}
C(f):\bigl(x,\{\be_i:i\in A\}\bigr)&\longmapsto
\bigl(z,\{\dot\be_j:j\in P^f,\; \Pi^f(j)\in A\}\bigr),\\
C(g):\bigl(y,\{\ti\be_i:i\in B\}\bigr)&\longmapsto
\bigl(z,\{\dot\be_j:j\in P^g,\; \Pi^g(j)\in B\}\bigr).
\end{split}
\label{mc8eq1}
\e

Lemma \ref{mc6lem} on $C(f),C(g)$ over $x,y,z$, which uses $f,g$
transverse, then turns out to be equivalent to the following
conditions on $P^f,P^g,\Pi^f,\Pi^g$:
\begin{itemize}
\setlength{\itemsep}{0pt}
\setlength{\parsep}{0pt}
\item[(A)] $\Pi^f(P^f\cap P^g)\cap
\Pi^f(P^f\sm P^g)=\es$ and $\Pi^g(P^f\cap P^g)\cap \Pi^g(P^g\sm
P^f)=\es$.
\item[(B)] $\Pi^f\vert_{P^f\sm P^g}\!:\!P^f\!\sm\! P^g\!\ra\!\{1,
\ldots,a\}$, $\Pi^g\vert_{P^g\sm P^f}\!:\!P^g\!\sm\!P^f\!\ra\!
\{1,\ldots,b\}$ are injective.
\item[(C)] Let $\approx$ be the equivalence relation on $P^f\cap P^g$
generated by $i\approx j$ if $\Pi^f(i)=\Pi^f(j)$ or
$\Pi^f(i)=\Pi^f(j)$. Then for each $\approx$-equivalence class
$E\subseteq P^f\cap P^g$ we have either
\begin{itemize}
\setlength{\itemsep}{0pt}
\setlength{\parsep}{0pt}
\item[(a)] $\md{\Pi^f(E)}+\md{\Pi^g(E)}=\md{E}+1$, or
\item[(b)] $\md{\Pi^f(E)}+\md{\Pi^g(E)}=\md{E}$.
\end{itemize}
Here it is automatic that
$\md{\Pi^f(E)}+\md{\Pi^g(E)}\le\md{E}+1$, and Lemma \ref{mc6lem}
implies that $\md{\Pi^f(E)}+\md{\Pi^g(E)} \ge\md{E}$. The number
of equivalence classes of type (a) is~$\md{\Pi^f(P^f\cap P^g)}+
\md{\Pi^g(P^f\cap P^g)}-\md{P^f\cap P^g}$.
\end{itemize}
Also, if $i\in\{1,\ldots,c\}\sm (P^f\cup P^g)$ then
$\bigl(x,(z,\dot\be_i)\bigr)\in\Xi^f_+$ and
$\bigl(y,(z,\dot\be_i)\bigr)\in\Xi^g_+$. These imply that $\d
f\vert_x(T_xX),\d g\vert_y(T_yY)\subseteq T_z\dot\be_i$, so that $\d
f\vert_x(T_xX)+\d g\vert_y(T_yY)\subseteq T_z\dot\be_i\subsetneq
T_zZ$, contradicting $f,g$ transverse. This proves:
\begin{itemize}
\setlength{\itemsep}{0pt}
\setlength{\parsep}{0pt}
\item[(D)] $P^f\cup P^g=\{1,\ldots,c\}$.
\end{itemize}

Now $V$ is cut out in $\hat V$ by $r_i\ge 0$, $i\le a$, and $s_i\ge
0$, $i\le b$, that is
\e
\begin{split}
V=\bigl\{\bigl((r_1,\ldots,r_k),(s_1,\ldots,s_l)\bigr)\in\hat V:\,&
\text{$r_i\ge 0$, $i=1,\ldots,a$, and}\\
&\text{$s_i\ge 0$, $i=1,\ldots,b$}\bigr\}.
\end{split}
\label{mc8eq2}
\e
We claim that making $R,S,\hat R,\hat S$ smaller if necessary, the
following hold:
\begin{itemize}
\setlength{\itemsep}{0pt}
\setlength{\parsep}{0pt}
\item[(i)] If $i\in P^f\sm P^g$ then the inequality $r_{\Pi^f(i)}\ge
0$ does not change $V$, and can be omitted in~\eq{mc8eq2}.
\item[(ii)] If $i\in P^g\sm P^f$ then the inequality
$s_{\Pi^g(i)}\ge 0$ does not change $V$, and can be omitted
in~\eq{mc8eq2}.
\item[(iii)] If $i,j\in P^f\cap P^g$ with $i\approx j$, then the four
inequalities $r_{\Pi^f(i)}\ge 0$, $r_{\Pi^f(j)}\ge 0$,
$s_{\Pi^g(i)}\ge 0$, $s_{\Pi^g(j)}\ge 0$ have the same effect in
$\hat V$. Thus for each $\approx$-equivalence class $E$ in
$P^f\cap P^g$, it is sufficient to impose only one of the
$2\md{E}$ inequalities $r_{\Pi^f(i)}\ge 0$, $s_{\Pi^g(i)}\ge 0$
in \eq{mc8eq2} for $i\in E$ to define~$V$.
\item[(iv)] If $E$ is an $\approx$-equivalence class of type (b) in
(C) above, then $r_{\smash{\Pi^f(i)}}\equiv s_{\Pi^g(i)}\equiv
0$ in $\hat V$ for all $i\in E$. Thus we can omit the
inequalities $r_{\Pi^f(i)}\ge 0$, $s_{\Pi^g(i)}\ge 0$ in
\eq{mc8eq2} for $i\in E$ to define~$V$.
\end{itemize}
Here we mean that all of the inequalities $r_i\ge 0$, $s_i\ge 0$
which (i)--(iv) allow us to omit may all be omitted simultaneously
without changing~$V$.

To prove (i), note that as $\bigl(\xi(T),t_i\ci\xi^{-1}\bigr)$ is a
boundary defining function for $Z$ at $(z,\dot\be_i)$, and
$\bigl(\th(R),r_{\Pi^f(i)}\ci\th^{-1}\bigr)$ is a boundary defining
function for $X$ at $(x,\be_{\Pi^f(i)})$, and
$\xi_-^f\bigl(x,(z,\dot\be_i)\bigr)=(x,\be_{\Pi^f(i)})$, Definition
\ref{mc3def1} and Proposition \ref{mc2prop6}(b) imply that
$t_i\ci\xi^{-1}\ci f\equiv (r_{\Pi^f(i)}\ci\th^{-1})\cdot G$ on
$\th(T)$ near $x$ for some smooth $G:\th(T)\ra(0,\iy)$. Hence
$t_i\ci\hat f\equiv r_{\Pi^f(i)}\cdot\hat G$ on $\hat R$ near 0 for
some smooth $\hat G:\hat R\ra(0,\iy)$ defined near 0. Therefore
making $R,\hat R$ smaller if necessary, $r_{\Pi^f(i)}\ge 0$ is
equivalent to $t_i\ci\hat f\ci\pi_{\hat R}\ge 0$ on $\hat V$. But
$t_i\ci\hat f\ci\pi_{\hat R}\!=\!t_i\ci\hat g\ci\pi_{\hat S}$, and
$t_i\!\ge\! 0$ on $T$, so $t_i\ci\hat g\!\ge\! 0$ on $S$ as $\hat g$
maps~$S\!\ra\! T$.

Hence the inequality $r_{\smash{\Pi^f(i)}}\ge 0$ is unnecessary in
\eq{mc8eq2} provided we restrict to $S$ in $\hat S$, that is,
provided we impose all the conditions $s_i\ge 0$. In fact we need
more than this: we must also be able to omit conditions
$s_{\Pi^g(j)}\ge 0$ when this is allowed by (ii)--(iv). This is
possible because the $s_{\Pi^g(j)}\ge 0$ omitted in (ii)--(iv)
correspond to different conditions $t_j\ge 0$ in $\hat T$ than the
condition $t_i\ge 0$ we are considering, since (i) deals with $i\in
P^f\sm P^g$, (ii) with $j\in P^g\sm P^f$, and (iii)--(iv) with $j\in
P^f\cap P^g$, which are disjoint sets. This proves (i), and also
that we can omit $r_{\Pi^f(i)}\ge 0$ in (i) independently of other
omissions in (i)--(iv). The proof for (ii) is the same.

For (iii), if $i\in P^f\cap P^g$ then by Definition \ref{mc3def1}
and Proposition \ref{mc2prop6}(b) as above we see that making
$R,\hat R,S,\hat S$ smaller if necessary, $r_{\smash{\Pi^f(i)}}\ge
0$ is equivalent to $t_i\ci\hat f\ci\pi_{\hat R}\ge 0$ on $\hat V$,
which is also equivalent to $s_{\smash{\Pi^g(i)}}\ge 0$ on $\hat V$.
Suppose $i,j\in P^f\cap P^g$ with $\Pi^f(i)=\Pi^f(j)$. Then the
conditions $r_{\smash{\Pi^f(i)}}\ge 0$, $r_{\smash{\Pi^f(j)}}\ge 0$
are the same, and are equivalent to $s_{\smash{\Pi^g(i)}}\ge 0$ and
$s_{\smash{\Pi^g(j)}}\ge 0$. Similarly, if $\Pi^g(i)=\Pi^g(j)$ then
the conditions $r_{\smash{\Pi^f(i)}}\ge 0$, $r_{\smash{\Pi^f(j)}}\ge
0$, $s_{\smash{\Pi^g(i)}}\ge 0$ and $s_{\smash{\Pi^g(j)}}\ge 0$ are
all equivalent. Since these two cases generate $\approx$, part (iii)
follows.

For (iv), let $E$ be an $\approx$-equivalence class of type (b).
Define submanifolds $\hat R_E,\hat S_E,\hat T_E$ in $\hat R,\hat
S,\hat T$ by
\begin{align*}
\hat R_E&=\bigl\{(r_1,\ldots,r_k)\in\hat R:r_{\smash{\Pi^f(i)}}=0,\;
i\in E\bigr\}, \\
\hat S_E&=\bigl\{(s_1,\ldots,s_l)\in\hat S:s_{\smash{\Pi^g(i)}}=0,\;
i\in E\bigr\}, \\
\hat T_E&=\bigl\{(t_1,\ldots,t_m)\in\hat T:t_i=0,\; i\in E\bigr\}.
\end{align*}
Making $R,S,\hat R,\hat S$ smaller if necessary, $\hat f$ maps $\hat
R_E\ra \hat T_E$ and $\hat g$ maps $\hat S_E\ra\hat T_E$. As $\hat
f,\hat g$ are transverse, an argument similar to Lemma \ref{mc6lem}
shows that $\hat f:\hat R_E\ra\hat T_E$ and $\hat g:\hat S_E\ra\hat
T_E$ are transverse, so the fibre product $\hat R_E\t_{\hat T_E}\hat
S_E$ exists as a manifold, and is a submanifold of $\hat V=\hat
R\t_{\hat T}\hat S$. Now
\begin{equation*}
\dim \hat R_E\!\t_{\hat T_E}\hat
S_E\!=\!(k\!-\!\md{\Pi^f(E)})\!+\!(l\!-\!\md{\Pi^g(E)})\!-\!
(m\!-\!\md{E})=k\!+\!l\!-\!m\!=\!\dim\hat V,
\end{equation*}
since $E$ is of type (b). Thus $\hat R_E\t_{\hat T_E}\hat S_E$ is
open in $\hat V$, as they are of the same dimension, and contains
$(0,0)$. So making $R,S,\hat R,\hat S$ smaller if necessary, we have
$\hat R_E\t_{\hat T_E}\hat S_E=\hat V$, proving (iv). There are no
further issues about simultaneous omissions in~(i)--(iv).

Choose a subset $Q\subseteq P^f\cap P^g$ such that $Q$ contains
exactly one element of each $\approx$-equivalence class of type (a)
in $P^f\cap P^g$, and no elements of $\approx$-equivalence classes
of type (b) . Then \eq{mc8eq2} and (i)--(iv) above imply that
\e
\begin{split}
V=\bigl\{\bigl((r_1,&\ldots,r_k),(s_1,\ldots,s_l)\bigr)\in\hat V:
r_i\ge 0,\; i\in\{1,\ldots,a\}\sm\Pi^f(P^f),\\
&s_i\ge 0,\; i\in\{1,\ldots,b\}\sm\Pi^g(P^g), \quad r_{\Pi^f(i)}\ge
0,\; i\in Q\bigr\}.
\end{split}
\label{mc8eq3}
\e
For the first condition $r_i\ge 0$, $i\in\{1,\ldots,a\}
\sm\Pi^f(P^f)$ in \eq{mc8eq3}, there are
\begin{equation*}
a-\bmd{\Pi^f(P^f)}=a-\bmd{\Pi^f(P^f\cap P^g)}-\bmd{\Pi^f(P^f\sm
P^g)}=a-\bmd{\Pi^f(P^f\cap P^g)}-\bmd{P^f\sm P^g}
\end{equation*}
inequalities, using (A) above in the first step and (B) in the
second. Similarly, for the second condition $s_i\ge 0$,
$i\in\{1,\ldots,b\} \sm\Pi^g(P^g)$ there are $b-\bmd{\Pi^g(P^f\cap
P^g)}-\bmd{P^g\sm P^f}$ inequalities. For the third there are
$\md{Q}=\md{\Pi^f(P^f\cap P^g)}+\md{\Pi^g(P^f\cap P^g)}-\md{P^f\cap
P^g}$ inequalities by (C) above. Hence in total there are
\e
\begin{split}
\bigl(&a-\bmd{\Pi^f(P^f\cap P^g)}-\bmd{P^f\sm P^g}\bigr)+
\bigl(b-\bmd{\Pi^g(P^f\cap P^g)}-\bmd{P^g\sm P^f}\bigr)\\
&\qquad+\bigl(\bmd{\Pi^f(P^f\cap P^g)}+\bmd{\Pi^f(P^f\cap
P^g)}-\md{P^f\cap P^g}\bigr)\\
&=\!a\!+\!b\!-\!\md{P^f\sm P^g}\!-\!\md{P^g\sm P^f}\!-\!\md{P^f\cap
P^g}\!=\!a\!+\!b\!-\!\md{P^f\cup P^g}\!=\!a\!+\!b\!-\!c\!\!
\end{split}
\label{mc8eq4}
\e
inequalities $r_i\ge 0$, $s_i\ge 0$ appearing in \eq{mc8eq3}, using
(D) at the last step.

Define a vector subspace $L$ of $T_{(0,0)}\hat V$ by
\e
\begin{split}
L&=\bigl\{\bigl((r_1,\ldots,r_k),(s_1,\ldots,s_l)\bigr)\in
T_{(0,0)}\hat V:\text{$r_i=0$, $i\le a$, $s_j=0$, $j\le b$}\bigr\}.
\end{split}
\label{mc8eq5}
\e
That is, we replace each inequality $r_i\ge 0$, $s_i\ge 0$ in
\eq{mc8eq2} by $r_i=0$, $s_i=0$. By the proof of the equivalence of
\eq{mc8eq2} and \eq{mc8eq3} using (i)--(iv), we see that
\e
\begin{split}
L&=\bigl\{\bigl((r_1,\ldots,r_k),(s_1,\ldots,s_l)\bigr)\in
T_{(0,0)}\hat V: r_{\Pi^f(i)}=0 ,\; i\in Q,\\
&\quad r_i=0,\; i\in\{1,\ldots,a\}\sm\Pi^f(P^f),\; s_i=0,\;
i\in\{1,\ldots,b\}\sm\Pi^g(P^g)\bigr\},
\end{split}
\label{mc8eq6}
\e
replacing inequalities $r_i\ge 0$, $s_i\ge 0$ in \eq{mc8eq3} by
$r_i=0$, $s_i=0$.

Now $T_z(S^c(Z))=\d f\vert_x(T_x(S^a(X)))+\d g\vert_y(T_y(S^b(Y)))$
by Definition \ref{mc6def1}. As there is a natural isomorphism
\begin{equation*}
L\cong \bigl\{u\op v\in T_x(S^a(X))\op T_y(S^b(Y)): \d f\vert_x(u)=
\d g\vert_y(v)\bigr\},
\end{equation*}
we see that
\e
\begin{split}
\dim L&=\dim S^a(x)+\dim S^b(Y)-\dim S^c(Z)\\
&=(p-a)+(q-b)-(r-c)=\dim\hat V-(a+b-c).
\end{split}
\label{mc8eq7}
\e
Since by \eq{mc8eq4} there are $a+b-c$ equalities $r_i=0$, $s_i=0$
in \eq{mc8eq6}, equation \eq{mc8eq7} implies that the conditions
$r_i=0$, $s_i=0$ in \eq{mc8eq6} are transverse, so the inequalities
$r_i\ge 0$, $s_i\ge 0$ in \eq{mc8eq3} are transverse at $(0,0)$ in
$\hat V$. That is, the corresponding 1-forms $\d r_i\vert_{(0,0)}$,
$\d s_i\vert_{(0,0)}$ are linearly independent at $T^*_{(0,0)}\hat
V$. This is an open condition in $\hat V$. Since $\hat V$ is a
manifold without boundary of dimension $n$, it follows that $V$ is
near $(0,0)$ a manifold with corners of dimension $n$, locally
modelled on $\R^n_{a+b-c}$. Making $R,S$ smaller if necessary, $V$
becomes an $n$-manifold with corners.

Let $(U,\phi')$ be a chart on $V$, with $U$ open in $\R^n_{a+b-c}$
and $0\in U$ with $\phi'(0)=(0,0)$. Define $\phi:U\ra W=X\t_ZY$ by
$\phi=(\th\t\psi)\ci\phi'$. Then $\phi$ is a homeomorphism with an
open set in $W$, since $\phi':U\ra V$ is a homeomorphism with an
open set in $V$ and $\th\t\psi:V\ra W$ is a homeomorphism with an
open set in $W$. Also $\phi(0)=(x,y)$ as $\phi'(0)=(0,0)$ and
$\th(0)=x$, $\psi(0)=y$. Thus $(U,\phi)$ is a chart on the
topological space $W$ whose image contains~$(x,y)$.

Now $(U,\phi')$ extends to a chart $(\hat U,\hat\phi')$ on $\hat V$.
But $\hat V$ comes from Theorem \ref{mc6thm1} for manifolds without
boundary, and so
\begin{equation*}
\d(\pi_{\hat S}\ci\hat\phi')\vert_{\hat u}\op\d(\pi_{\hat T}\ci\hat
\phi')\vert_{\hat u}:T_{\hat u}\hat U\longra T_{\hat s}\hat S\op
T_{\hat t}\hat T
\end{equation*}
is injective for all $\hat u\in\hat U$ with $\hat\phi'(\hat u)=(\hat
s,\hat t)$. Restricting to $U$ shows that
\e
\d(\pi_S\ci\phi')\vert_u\op\d(\pi_T\ci\phi')\vert_u:T_uU\longra T_sS
\op T_tT
\label{mc8eq8}
\e
is injective for all $u\in U$ with $\phi'(u)=(s,t)$. But if $u\in U$
with $\phi'(u)=(s,t)$ and $\th(s)=x'$, $\psi(t)=y'$ then
$\d(\pi_X\ci\phi)\vert_u=\d\th\vert_s\ci\d(\pi_S\ci\phi')\vert_u$
and $\d(\pi_Y\ci\phi)\vert_u=\d\psi\vert_t\ci\d(\pi_T\ci\phi')
\vert_u$, where $\d\th\vert_s:T_sS\ra T_{x'}X$ and
$\d\psi\vert_t:T_tT\ra T_{y'}Y$ are isomorphisms as
$(S,\th),(T,\psi)$ are charts on $X,Y$. So composing \eq{mc8eq8}
with $\d\th\vert_s\t \d\psi\vert_t$ shows that
\begin{equation*}
\d(\pi_X\ci\phi)\vert_u\op\d(\pi_Y\ci\phi)\vert_u:T_uU\longra
T_{x'}X\op T_{y'}Y
\end{equation*}
is injective, so $(U,\phi)$ satisfies the conditions of
Theorem~\ref{mc6thm1}.

We have now shown that $W$ can be covered by charts $(U,\phi)$
satisfying the conditions of Theorem \ref{mc6thm1}. For such
$(U,\phi)$, observe that \eq{mc6eq1} actually maps
\e
\begin{split}
\d(\pi_X\ci\phi)_{u}&\op\d(\pi_Y\ci\phi)_{u}:
T_{u}U\cong\R^n\longra\\
&\bigl\{(\al,\be)\in T_xX\op T_yY:\d f\vert_x(\al)=\d g\vert_y(\be)
\bigr\}.
\end{split}
\label{mc8eq9}
\e
Now the r.h.s.\ of \eq{mc8eq9} has dimension $n$ by transversality
of $f,g$, and \eq{mc8eq9} is injective, so it is an isomorphism. Let
$(U,\phi)$ and $(V,\psi)$ be two such charts, and $u\in U$, $v\in V$
with $\phi(u)=\psi(v)=(x,y)$. Since \eq{mc8eq9} and its analogue for
$(V,\psi)$ are isomorphisms, we see that $\psi^{-1} \ci\phi$ is
differentiable at $u$ and its derivative is an isomorphism, the
composition of \eq{mc8eq9} with the inverse of its analogue for
$\psi$. Using the same argument for all $u\in \phi^{-1}(\psi(V))$,
we find that $\psi^{-1}\ci\phi:\phi^{-1}(\psi(V))\ra
\psi^{-1}(\phi(U))$ is a diffeomorphism, and so $(U,\phi),(V,\psi)$
are automatically compatible.

Therefore the collection of all charts on $W$ satisfying the
conditions of Theorem \ref{mc6thm1} is an atlas. But any chart
compatible with all charts satisfying Theorem \ref{mc6thm1} also
satisfies the conditions of Theorem \ref{mc6thm1}, so this atlas is
maximal. Also the topological space $W=X\t_ZY$ is paracompact and
Hausdorff, since $X,Y,Z$ are paracompact and Hausdorff as they are
manifolds. Hence the construction of Theorem \ref{mc6thm1} does make
$W$ into an $n$-manifold with corners.

It remains to show that $\pi_X:W\ra X,$ $\pi_Y:W\ra Y$ are smooth.
They are clearly continuous, since $W$ was defined as the
topological fibre product. Locally $\pi_X,\pi_Y$ are identified with
$\pi_S:V\ra S$ and $\pi_T:V\ra T$ above, which are restrictions to
$V$ of $\pi_{\hat S}:\hat V\ra\hat S$ and $\pi_{\hat T}:\hat
V\ra\hat T$. But $\hat V$ is a fibre product of manifolds without
boundary, so $\pi_{\hat S},\pi_{\hat T}$ are smooth, which implies
that $\pi_S,\pi_T$ are weakly smooth. To prove $\pi_S,\pi_T$ are
smooth we note that $(S,s_i)$ for $i=1,\ldots,a$ are boundary
defining functions on $S$ at $(s,\{s_i=0\})$, and $(T,t_j)$ for
$j=1,\ldots,b$ are boundary defining functions on $T$ at
$(t,\{t_j=0\})$, and we show using the discussion of (i)--(iv) that
the pullbacks to $V$ satisfy the conditions of Definition
\ref{mc3def1}. As smoothness is a local condition, $\pi_X,\pi_Y$ are
smooth.
\end{proof}

\begin{prop} In the situation of Proposition\/ {\rm\ref{mc8prop1},}
$W,\pi_X,\pi_Y$ are a fibre product\/ $X\t_{f,Z,g}Y$ in\/~$\Manc$.
\label{mc8prop2}
\end{prop}

\begin{proof} By definition $f\ci\pi_X=g\ci\pi_Y$. Suppose $W'$ is
a manifold with corners and $\pi_X':W'\ra X$, $\pi_Y':Y'\ra Y$ are
smooth maps with $f\ci\pi_X'=g\ci\pi_Y'$. We must show there exists
a unique smooth map $h:W'\ra W$ with $\pi_X'=\pi_X\ci h$ and
$\pi_Y'=\pi_Y\ci h$. Since $W$ is a fibre product at the level of
topological spaces, there is a unique continuous map $h:W'\ra W$
given by $h(w')=\bigl(\pi_X'(w'),\pi_Y'(w')\bigr)$ with
$\pi_X'=\pi_X\ci h$ and $\pi_Y'=\pi_Y\ci h$. We must show $h$ is
smooth.

Let $w'\in W'$, with $\pi'_X(w')=x\in X$ and $\pi'_Y(w')=y\in Y$, so
that $f(x)=g(y)=z\in Z$. Let $(R,\th),(S,\psi),\ldots$ be as in the
proof of Proposition \ref{mc8prop1}. Then $\hat V=\hat R\t_{\hat
T}\hat S$ is a fibre product of manifolds without boundary, and
$V=R\t_TS\subseteq\hat V$ is a manifold with corners, and
$\th\t\psi:V\ra W$ is a diffeomorphism with an open subset of $W$.
Let $U'$ be an open neighbourhood of $w'\in W'$ such that
$\pi_X'(U')\subseteq\th(S)$ and $\pi_Y'(U')\subseteq\psi(T)$.
Consider the map $\ti h=(\th\t\psi)^{-1}\ci h\vert_{U'}:U'\ra
V\subseteq\hat V$. We will show $\ti h$ is smooth. This implies
$h\vert_{U'}=(\th\t\psi)\ci\ti h$ is smooth, so $h$ is smooth as
this is a local condition.

As $\pi_{\hat S}\ci\ti h=\th^{-1}\ci\pi'_X$ and $\pi_{\hat T}\ci\ti
h=\psi^{-1}\ci\pi'_Y$ are smooth, and $\hat V=\hat R\t_{\hat T}\hat
S$ is a fibre product of manifolds, we see that $\ti h:U'\ra \hat V$
is smooth, and therefore $\ti h:U'\ra V$ is weakly smooth. To show
$\ti h:U'\ra V$ is smooth, we must verify the additional condition
in Definition \ref{mc3def1}. It is enough to do this at $w'\in W$
and $(0,0)\in V$. The proof of Proposition \ref{mc8prop1} shows that
$V$ is given by \eq{mc8eq3}, and the inequalities $r_i\ge 0$,
$s_i\ge 0$ in \eq{mc8eq3} are transverse. Therefore, if $\be$ is a
local boundary component of $V$ at $(0,0)$, then either (a)
$(V,r_i)$ is a boundary defining function for $V$ at $((0,0),\be)$
for some $r_i\ge 0$ appearing in \eq{mc8eq3}, or (b) $(V,s_i)$ is a
boundary defining function for $V$ at $((0,0),\be)$ for some $s_i\ge
0$ in~\eq{mc8eq3}.

In case (a), as $\bigl(\th(U), r_i\ci\th^{-1}\bigr)$ is a local
boundary defining function for $X$ at $(x,\th_*(\{r_i=0\})$, and
$\pi_X':W'\ra X$ is smooth, by Definition \ref{mc3def1} either
$r_i\ci\th^{-1}\ci\pi_X'\equiv 0$ near $w'$ in $W$ or
$\bigl((\pi_X')^{-1}(\th(U)), r_i\ci\th^{-1}\ci\pi_X'\bigr)$ is a
boundary defining function for $W'$ at some $(w',\ti\be)$. Since
$r_i\ci\ti h=r_i\ci\th^{-1}\ci\pi'_X\vert_{U'}$ and $U'$ is an open
neighbourhood of $w'$ in $W$ it follows that either $r_i\ci\ti
h\equiv 0$ near $w'$, or $(V,r_i\ci\ti h)$ is a boundary defining
function for $U'$ at some $(w',\ti\be)$. This proves the additional
condition in Definition \ref{mc3def1} in case (a). The proof for (b)
is the same, using $\pi_Y':W'\ra Y$ smooth. Thus $\ti h$, and hence
$h$, is smooth.
\end{proof}

\section{Proof of Theorem \ref{mc6thm2}}
\label{mc9}

We first construct bijections \eq{mc6eq9}--\eq{mc6eq10}. Let $X,Y,Z$
be manifolds with corners, and $f:X\ra Z,$ $g:Y\ra Z$ be strongly
transverse smooth maps, and write $W$ for the fibre product
$X\t_{f,Z,g}Y$, which we proved exists as a manifold with corners in
\S\ref{mc8}. Let $(x,y)\in W$, so that $x\in X$ and $y\in Y$ with
$f(x)=g(y)=z\in Z$. Use all the notation of Proposition
\ref{mc8prop1}, so that the local boundary components of $X$ at $x$
are $\be_1,\ldots,\be_a$, of $Y$ at $y$ are $\ti\be_1,\ldots,
\ti\be_b$, and of $Z$ at $z$ are $\dot\be_1,\ldots,\dot\be_c$. Then
over $x,y,z$, the maps $C(f),C(g)$ are given explicitly by
\eq{mc8eq1} in terms of $P^f,P^g\subseteq\{1,\ldots,c\}$ and maps
$\Pi^f:P^f\ra\{1,\ldots,a\}$ and~$\Pi^g:P^g\ra\{1,\ldots,b\}$.

Properties of these $P^f,P^g,\Pi^f,\Pi^g$ are given in (A)--(D) of
the proof of Proposition \ref{mc8prop1}. In addition, as $f,g$ are
strongly transverse, there are no $\approx$-equivalence classes $E$
of type (b) in part (C), since if $E$ is such a class then
\eq{mc8eq1} gives
\e
C(f)\bigl(x,\{\be_i:i\!\in\! \Pi^f(E)\}\bigr)\!=\!
C(g)\bigl(y,\{\ti\be_i:i\!\in\! \Pi^g(E)\}\bigr)\!=\!
\bigl(z,\{\dot\be_j:j\!\in\! E\}\bigr),
\label{mc9eq1}
\e
and $j=\md{\Pi^f(E)}$, $k=\md{\Pi^g(E)}$, $l=\md{E}$ satisfy
$j,k,l>0$ and $j+k=l$, contradicting Definition \ref{mc6def2}. Using
\eq{mc8eq1} and these properties of $P^f,P^g,\Pi^f,\Pi^g$ we can
describe the points of the r.h.s.\ of \eq{mc6eq9} over $x,y,z$
explicitly when $i=1$. We divide such points into three types:
\begin{itemize}
\setlength{\itemsep}{0pt}
\setlength{\parsep}{0pt}
\item[(i)] $\bigl((x,\{\be_i\}),(y,\es)\bigr)$ for
$i\in\{1,\ldots,a\}\sm\Pi^f(P^f)$ lies in the term on the
r.h.s.\ of \eq{mc6eq9} with $i\!=\!j\!=\!1$ and $k\!=\!l\!=\!0$,
as $C(f)(x,\{\be_i\})=C(g)(y,\es)=(z,\es)$.
\item[(ii)] $\bigl((x,\es),(y,\{\ti\be_i\})\bigr)$ for
$i\in\{1,\ldots,b\}\sm\Pi^g(P^g)$ lies in the term on the
r.h.s.\ of \eq{mc6eq9} with $i\!=\!k\!=\!1$ and $j\!=\!l\!=\!0$,
as $C(f)(x,\es)=C(g)(y,\{\ti\be_i\})=(z,\es)$.
\item[(iii)] $\bigl((x,\{\be_i:i\in \Pi^f(E)\}),(y,\{\ti\be_i:i\in
\Pi^g(E)\})\bigr)$ for $E$ a $\approx$-equivalence class of type
(a) in part (C) lies in the term on the r.h.s.\ of \eq{mc6eq9}
with $i=1$, $j=\md{\Pi^f(E)}$, $k=\md{\Pi^g(E)}$, $l=\md{E}$,
since then \eq{mc9eq1} holds, and $i=1=j+k-l$.
\end{itemize}

Now $W$ near $(x,y)$ is diffeomorphic to $V$ near $(0,0)$, where $V$
is given in \eq{mc8eq3}, and the inequalities $r_i\ge 0$, $s_i\ge 0$
in \eq{mc8eq3} are transverse. Thus, the local boundary components
of $W$ at $(x,y)$ correspond to the inequalities $r_i\ge 0$, $s_i\ge
0$ appearing in \eq{mc8eq3}. These in turn correspond to points in
the r.h.s.\ of \eq{mc6eq9} with $i=1$ as follows:
\begin{itemize}
\setlength{\itemsep}{0pt}
\setlength{\parsep}{0pt}
\item The local boundary component $r_i=0$ for
$i\in\{1,\ldots,a\}\sm\Pi^f(P^f)$ of $V$ at $(0,0)$ corresponds to
the point $\bigl((x,\{\be_i\}),(y,\es)\bigr)$ of type (i).
\item The local boundary component $s_i=0$ for
$i\in\{1,\ldots,b\}\sm\Pi^g(P^g)$ of $V$ at $(0,0)$ corresponds to
the point $\bigl((x,\es),(y,\{\ti\be_i\})\bigr)$ of type (ii).
\item The local boundary component $r_{\smash{\Pi^f(i)}}=0$ for $i\in
Q$ of $V$ at $(0,0)$ corresponds to the point $\bigl((x,\{\be_i:i\in
\Pi^f(E)\}),(y,\{\ti\be_i:i\in \Pi^g(E)\})\bigr)$ of type (iii),
where $E$ is the unique $\approx$-equivalence class containing $i$.

(Note that $Q$ contains one element of each $\approx$-equivalence
class.)
\end{itemize}
The natural map $\coprod_{i\ge 0}C(W)\ra \coprod_{j\ge
0}C(X)\t_{\coprod_{l\ge 0}C(Z)}\coprod_{k\ge 0}C(Y)$ referred to in
the last part of Theorem \ref{mc6thm2} agrees with this
correspondence.

This proves that \eq{mc6eq9} is a bijection for $i=1$. For the
general case, suppose $(w,\{\hat\be_1,\ldots,\hat\be_i\})\in
C_i(W)$. Let $(w,\{\hat\be_a\})$ correspond to
$\bigl((x,J_a),(y,K_a)\bigr)$ as above for $a=1,\ldots,i$. Then
$C(f)(x,J_a)=C(g)(y,K_a)=(z,L_a)$ for some $L_a$. It is easy to show
that as $\hat\be_1,\ldots,\hat\be_i$ are distinct, the subsets
$J_1,\ldots,J_i$ are disjoint, and $K_1,\ldots,K_i$ are disjoint,
and $L_1,\ldots,L_i$ are disjoint. Also $C(
f)(x,J_1\amalg\cdots\amalg J_i)=C(g)(y,K_1\amalg\cdots\amalg
K_i)=(z,L_1\amalg\cdots\amalg L_i)$. Hence
\begin{align*}
&\bigl((x,J_1\amalg\cdots\amalg J_i),(y,K_1\amalg\cdots\amalg
K_i)\bigr)\in \\
&\bigl(C_j(X)\cap C(f)^{-1}(C_l(Z))\bigr)\t_{C(f),C_l(Z),C(g)}\!
\bigl(C_k(Y)\cap C(g)^{-1}(C_l(Z))\bigr),
\end{align*}
where $j=\md{J_1}+\cdots+\md{J_i}$, $k=\md{K_1}+\cdots+\md{K_i}$ and
$l=\md{L_1}+\cdots+\md{L_i}$. As $1=\md{J_a}+\md{K_a}-\md{L_a}$ for
$a=1,\ldots,i$, we have $i=j+k-l$. So mapping $(w,\{\hat\be_1,
\ldots,\hat\be_i\})$ to $\bigl((x,J_1\amalg\cdots\amalg J_i),(y,K_1
\amalg\cdots\amalg K_i)\bigr)$ takes the l.h.s.\ of \eq{mc6eq9} to
the r.h.s.\ of \eq{mc6eq9}. Generalizing the argument in the $i=1$
case proves that this map is a bijection, so \eq{mc6eq9} is a
bijection, and thus \eq{mc6eq10} is a bijection.

Now let $(w,\{\hat\be_{a_1},\ldots,\hat\be_{a_i}\})\in C_i(W)$ with
\begin{align*}
C(\pi_X):(w,\{\hat\be_{a_1},\ldots,\hat\be_{a_i}\})&\longmapsto
(x,\{\be_{b_1},\ab\ldots,\ab\be_{b_j}\})\in C_j(X),\\
C(\pi_Y):(w,\{\hat\be_{a_1},\ldots,\hat\be_{a_i}\})&\longmapsto
(y,\{\ti\be_{c_1},\ldots,\ti\be_{c_k}\})\in C_k(Y),\\
C(f):(x,\{\be_{b_1},\ldots,\be_{b_j}\})&\longmapsto
(z,\{\dot\be_{d_1},\ldots,\dot\be_{d_l}\})\in C_l(Z),
\quad\text{and}\\
C(g):\bigl(y,\{\ti\be_{c_1},\ldots,\ti\be_{c_k}\}\bigr)&\longmapsto
(z,\{\dot\be_{d_1},\ldots,\dot\be_{d_l}\})\in C_l(Z).
\end{align*}
Define an immersion $\io^i_W:C_i(W)\ra W$ by
$\io^i_W:(w,\{\hat\be_{a_1},\ldots,\hat\be_{a_i}\})\mapsto w$, and
similarly for $X,Y,Z$. Then $i_Z^l\ci C(f)\equiv f\ci i_X^j$ near
$(x,\{\be_{b_1},\ldots,\be_{b_j}\})$ in $C_j(X)$, and so on, so we
have a commutative diagram
\e
\begin{gathered}
\xymatrix@R=5pt@C=17pt{
T_{(w,\{\hat\be_{a_1},\ldots,\hat\be_{a_i}\})}C_i(W) \ar[rr]_{\d
C(\pi_Y)} \ar[dd]_{\d C(\pi_X)} \ar[dr]_(0.65){\d \io_W^i} &&
T_{(y,\{\ti\be_{c_1},\ldots,\ti\be_{c_k}\})}C_k(Y)
\ar[dd]^(0.3){\d C(g)} \ar[dr]^(0.6){\d\io_Y^k} \\
& T_wW \ar[rr]^(0.2){\d\pi_Y} \ar[dd]^(0.3){\d\pi_X}
&& T_yY \ar[dd]_(0.45){\d g} \\
T_{(x,\{\be_{b_1},\ldots,\be_{b_j}\})}C_j(X) \ar[rr]^(0.4){\d C(f)}
\ar[dr]_(0.6){\d\io_X^j} && T_{(z,\{\dot\be_{d_1},\ldots,
\dot\be_{d_l}\})}C_l(Z) \ar[dr]^(0.65){\d\io_Z^l} \\
& T_xX \ar[rr]^{\d f} && T_zZ.}
\end{gathered}
\label{mc9eq2}
\e

Since $W=X\t_ZY$ in $\Manc$, in \eq{mc9eq2} we have an isomorphism
\e
\d\pi_X\op\d\pi_Y: T_wW\,{\buildrel\cong\over\longra}\,
\Ker\bigl((\d f\op -\d g):T_xX\op T_yY\longra T_zZ\bigr).
\label{mc9eq3}
\e
As $\io_W^i,\io_X^j,\io_Y^k,\io_Z^l$ are immersions, the diagonal
maps $\d\io_W^i,\d\io_X^j,\d\io_Y^k,\d\io_Z^l$ in \eq{mc9eq2} are
injective. Thus we can identify $T_{(w,\{\hat\be_{a_1},\ldots,
\hat\be_{a_i}\})}C_i(W)$ with its image in $T_wW$ under $\d\io_W^i$,
and similarly for $X,Y,Z$. The proof of Proposition \ref{mc8prop1}
implies that for each local boundary component $\hat\be_a$ of $W$ at
$w$, the tangent space $T_w\hat\be_a\subset T_wW$ is the pullback
under $\d f$ or $\d g$ of $T_x\be_b$ or $T_y\ti\be_c$ for
appropriate local boundary components $\be_b$ of $X$ at $x$ or
$\ti\be_c$ of $Y$ at $y$. So using that \eq{mc9eq3} is an
isomorphism and $\d\io_W^i$ is injective, we see that
\e
\begin{split}
&\d C(\pi_X)\op\d
C(\pi_Y):T_{(w,\{\hat\be_{a_1},\ldots,\hat\be_{a_i}\})}
C_i(W) \longra\\
&\Ker\bigl(\d C(f)\op -\d C(g):T_{(x,\{\be_{b_1},\ldots,\be_{b_j}
\})}C_j(X)\op T_{(y,\{\ti\be_{c_1},\ldots,\ti\be_{c_k}\})}C_k(Y)\\
&\qquad\qquad\qquad\qquad \longra T_{(z,\{\dot\be_{d_1},\ldots,
\dot\be_{d_l}\})}C_l(Z)\bigr)
\end{split}
\label{mc9eq4}
\e
is an isomorphism. That is, \eq{mc9eq4} is injective as it is a
restriction of \eq{mc9eq3} which is injective, and it is surjective
as the equations defining
$T_{(w,\{\hat\be_{a_1},\ldots,\hat\be_{a_i}\})} C_i(W)$ in $T_wW$
are pullbacks of equations defining
$T_{(x,\{\be_{b_1},\ldots,\be_{b_j}\})} C_j(X)$ in $T_xX$ or
defining $T_{(y,\{\ti\be_{c_1},\ldots,\ti\be_{c_k}\})}C_k(Y)$ in
$T_yY$.

As $W=X\t_ZY$ we have $\dim W=\dim X+\dim Y-\dim Z$, and $i=j+k-l$
from above. So $\dim C_i(W)=\dim C_j(X)+\dim C_k(Y)-\dim C_l(Z)$ as
$\dim C_i(W)=\dim W-i$,\ldots. Together with \eq{mc9eq4} an
isomorphism this implies
\e
\begin{split}
&T_{(z,\{\dot\be_{d_1},\ldots, \dot\be_{d_l}\})}C_l(Z)= \\
&\qquad\d C(f)\bigl(T_{(x,\{\be_{b_1},\ldots,\be_{b_j} \})}C_j(X)
\bigr)+\d
C(g)\bigl(T_{(y,\{\ti\be_{c_1},\ldots,\ti\be_{c_k}\})}C_k(Y) \bigr).
\end{split}
\label{mc9eq5}
\e

Let $x\in S^p(X)$, $y\in S^q(Y)$ and $z\in S^r(Z)$. Then as $f$ is
transverse we have
\e
T_z(S^r(Z))=\d f\vert_x(T_x(S^p(X)))+\d g\vert_y(T_y(S^q(Y)))
\label{mc9eq6}
\e
Clearly $(x,\{\be_{b_1},\ldots,\be_{b_j}\})\in S^{p-j}(C_j(X))$ and
$\d\io^j_X\bigl(T_{(x,\{\be_{b_1},\ldots,\be_{b_j}\})}
S^{p-j}C_j(X)\bigr)\ab =T_xS^p(X)$. So pulling back \eq{mc9eq6}
using $\io^j_X,\io^k_Y,\io^l_Z$ yields
\e
\begin{split}
T_{(z,\{\dot\be_{d_1},\ldots,\dot\be_{d_l}\})}&(S^{r-l}(C_l(Z)))=\\
&\d C(f)\vert_{(x,\{\be_{b_1},\ldots,\be_{b_j}\})}
(T_{(x,\{\be_{b_1},\ldots,\be_{b_j}\})} (S^{p-j}(C_j(X))))+\\
&\d C(g)\vert_{(y,\{\ti\be_{c_1},\ldots,\ti\be_{c_k}\})}
(T_{(y,\{\ti\be_{c_1},\ldots,\ti\be_{c_k}\})}(S^{q-k}(C_k(Y)))).
\end{split}
\label{mc9eq7}
\e
Equations \eq{mc9eq5} and \eq{mc9eq7} imply that the fibre product
in \eq{mc6eq9} is transverse, as we have to prove. So the fibre
products in \eq{mc6eq9} exist in $\Manc$ by Theorem \ref{mc6thm1},
and the natural map from the left hand side of \eq{mc6eq9} to the
right hand side is smooth. We have already shown it is a bijection,
and \eq{mc9eq4} an isomorphism implies that this natural map induces
isomorphisms on tangent spaces. Therefore \eq{mc6eq9} is a
diffeomorphism. This completes the proof of Theorem~\ref{mc6thm2}.

\medskip

\noindent{\small\sc The Mathematical Institute, 24-29 St. Giles,
Oxford, OX1 3LB, U.K.}

\noindent{\small\sc E-mail: \tt joyce@maths.ox.ac.uk}


\begin{thebibliography}{99}

\bibitem{Cerf} J. Cerf, {\it Topologie de certains espaces de
plongements}, Bull. Soc. Math. France 89 (1961), 227--380.

\bibitem{Doua} A. Douady, {\it Vari\'et\'es \`a bord anguleux et
voisinages tubulaires}, S\'eminaire Henri Cartan 14 (1961-2), exp.
1, 1--11.

\bibitem{FOOO} K. Fukaya, Y.-G. Oh, H. Ohta and K. Ono, {\it
Lagrangian intersection Floer theory -- anomaly and obstruction},
preprint, final version, 2008. 1385 pages.

\bibitem{HWZ} H. Hofer, K. Wysocki and E. Zehnder, {\it A general
Fredholm theory I: A splicing-based differential geometry}, J. Eur.
Math. Soc. 9 (2007), 841--876. math.FA/0612604.

\bibitem{Jani} K. J\"anich, {\it On the classification of\/
$O(n)$-manifolds}, Math. Ann. 176 (1968), 53--76.

\bibitem{Joyc} D. Joyce, {\it D-manifolds and d-orbifolds: a
theory of derived differential geometry}, in preparation, 2010.

\bibitem{KoNo} S. Kobayashi and K. Nomizu, {\it Foundations of
Differential Geometry, volume I}, John Wiley \& sons, New York,
1963.

\bibitem{Lang} S. Lang, {\it Differentiable manifolds},
Addison-Wesley, Reading, MA, 1972.

\bibitem{LaPf}  Lauda, A.D. and Pfeiffer, H., {\it Open-closed
strings: two-dimensional extended TQFTs and Frobenius algebras},
Topology Appl. 155 (2008), 623--666. math.AT/0510664.

\bibitem{Laur} G. Laures, {\it On cobordism of manifolds with
corners}, Trans. A.M.S. 352 (2000), 5667--5688.

\bibitem{MaOu} J. Margalef-Roig and E. Outerelo Dominguez, {\it
Differential Topology}, North-Holland Math. Studies 173,
North-Holland, Amsterdam, 1992.

\bibitem{Melr1} R.B. Melrose, {\it The Atiyah--Patodi--Singer Index
Theorem}, A.K. Peters, Wellesley, MA, 1993.

\bibitem{Melr2} R.B. Melrose, {\it Differential Analysis on
Manifolds with Corners}, unfinished book available at {\tt
http://math.mit.edu/$\sim$rbm}, 1996.

\bibitem{Mont} B. Monthubert, {\it Groupoids and pseudodifferential
calculus on manifolds with corners}, J. Funct. Anal. 199 (2003),
243--286.

\bibitem{RaBa} L. Ramshaw and J. Basch, {\it Orienting transverse
fibre products}, Hewlett Packard Technical Report HPL-2009-144,
2009.

\bibitem{Spiv} D.I. Spivak, {\it Derived smooth manifolds},
arXiv:0810.5174, 2008. To appear in Duke Mathematical Journal.

\end{thebibliography}
\end{document}